\documentclass{fundam}

\usepackage{amsmath,amsfonts,amssymb, mathtools, latexsym, stmaryrd}
\usepackage[T1]{fontenc}
\usepackage{mathabx}
\usepackage[all,color]{xy}
\usepackage{rotating}
\usepackage{enumitem}

\newcommand{\rad}{\mathbf{rad}}

\newcommand{\Rnk}{\mathbf{rk}}

\newcommand{\tr}{\mathbf{tr}}
\newcommand{\Dyn}{\mathbf{Dyn}}
\newcommand{\Z}{\mathbb{Z}}

\newcommand{\Inc}{\mathbf{Inc}}
\newcommand{\A}{\mathbb{A}}
\newcommand{\B}{\mathbb{B}}
\newcommand{\C}{\mathbb{C}}
\newcommand{\D}{\mathbb{D}}
\newcommand{\E}{\mathbb{E}}
\newcommand{\F}{\mathbb{F}}
\newcommand{\G}{\mathbb{G}}
\newcommand{\CRnk}{\mathbf{cork}}

\newcommand{\Id}{\mathbf{I}}

\newcommand{\Adj}{\widecheck{\mathrm{Adj}}}

\newcommand{\FS}{\widetilde{\mathcal{T}}}
\newcommand{\FSq}{\widetilde{T}}
\newcommand{\Star}{\mathbb{S}}
\newcommand{\bas}{\mathbf{e}}
\newcommand{\sgn}{\mathrm{sgn}}
\newcommand{\sou}{\mathbf{s}}
\newcommand{\tar}{\mathbf{t}}
\newcommand{\mi}[1]{\hat{#1}}
\newcommand{\Cox}{\mathrm{\Phi}}
\newcommand{\cox}{\mathrm{\varphi}}
\newcommand{\coxN}{\mathbf{c}}
\newcommand{\va}{\lambda}

\newcommand{\MiNu}{\mathbf{v}}

\newcommand\mathmiddlescript[1]{\vcenter{\hbox{$\scriptstyle #1$}}}

   \usepackage{hyperref}

\begin{document}

\setcounter{page}{49}
\publyear{2021}
\papernumber{2092}
\volume{184}
\issue{1}

   \finalVersionForARXIV

\title{A Graph Theoretical Framework for the Strong Gram Classification of  Non-negative Unit Forms of Dynkin Type $\A_n$}

\author{Jes\'us Arturo~Jim\'enez Gonz\'alez\thanks{Address of correspondence: Instituto de Matem\'aticas, Mexico City, Mexico. \newline \newline
          \vspace*{-6mm}{\scriptsize{Received June 2021; \ revised November 2021.}}}
\\
 Instituto de Matem\'aticas, UNAM, Mexico\\
jejim@im.unam.mx
}

\maketitle

\runninghead{J.A. Jim\'enez Gonz\'alez}{A Graph Theoretical Framework for the Strong Gram Classification of  Non-negative}

\begin{abstract}
In the context of signed line graphs, this article introduces a modified inflation technique to study strong Gram congruence of non-negative (integral quadratic) unit forms, and uses it to show  that weak and strong Gram congruence coincide among positive unit forms of Dynkin type $\A_n$. The concept of inverse of a quiver is also introduced, and is used to obtain and analyze the Coxeter matrix of non-negative unit forms of Dynkin type $\A_n$. With these tools, connected principal unit forms of Dynkin type $\A_n$ are also classified up to strong congruence.
\end{abstract}


\section{Introduction} \label{Intro}

An integral quadratic form is a homogeneous polynomial $q$ of degree two with integer coefficients, given usually as $q(x)=x^{\tr}\widecheck{G}_qx$ for a unique upper triangular matrix $\widecheck{G}_q$. In case $\widecheck{G}_q$ has only $1$'s as diagonal entries, $q$ is simply called a unit form, and the symmetric matrix $G_q=\widecheck{G}_q+\widecheck{G}_q^{\tr}$ is a \emph{generalized Cartan matrix} (see for instance~\cite{BGZ06} or~\cite{dS20}). Two unit forms $q'$ and $q$ are said to be weakly (resp. strongly) Gram congruent, if there is an integer matrix $B$ with $\det(B)=\pm 1$ such that $G_{q'}=B^{\tr}G_qB$ (resp. $\widecheck{G}_{q'}=B^{\tr}\widecheck{G}_qG$). Clearly, strong congruence implies weak congruence.

The classification of connected positive unit forms up to weak Gram congruence is well known (see Ovsienko~\cite{saO78}, Kosakowska~\cite{jK12} and Simson~\cite{dS13}). Corresponding generalizations to the non-negative case are also known (see Barot and de la Pe\~na~\cite{BP99,BP06}, and Simson et al.~\cite{dS16a,SZ17}). A strong Gram classification of non-negative unit forms is far from been completed (see~\cite[Problem~2.1$(a)$]{dS18} and~\cite[Problem~1.12]{dS20} for a specific formulation of these problems in terms of Coxeter spectra):

\medskip\noindent
\textbf{Problem A.}\ 
Classify all (connected) non-negative unit forms up to strong Gram congruence.

\medskip
The following problem was posed by Simson in~\cite[Problems~1.10 and~1.11]{dS13b} (see also~\cite[Problem~2.1$(b)$]{dS18}):

\medskip\smallskip\noindent
\textbf{Problem B.}\
Construct algorithms that compute an integer matrix $B$ with $\det(B)=\pm 1$ that defines the strong Gram congruence $\widecheck{G}_{q'}=B^{\tr}\widecheck{G}_{q}B$, in case the quadratic unit forms $q$ and $q'$ are strongly Gram congruent.

\medskip\smallskip
There have been many advances towards the strong classification of positive quadratic forms, both from computational and geometrical points of view. For instance, a classification for small cases ($n \leq 9$), including the exceptional cases $\E_6$, $\E_7$ and $\E_8$, as well as all non-simply laced cases, was presented by Simson et al., cf.~\cite{dS20}. The general case of Dynkin type $\D_n$ was announced in~\cite{dS18} (see also~\cite[\S 4]{dS20}), and is solved by Simson in~\cite{dS21a}.

There is a well known graphical description of quadratic (unit) forms by means of (loop-less) signed multi-graphs (or bigraphs, as called in the paper following Simson~\cite{dS13}). It is easy to verify that if $q(x)=\frac{1}{2}x^{\tr}(2\Id+A)x$, where $A$ is the symmetric adjacency matrix of a loop-less bigraph (where $\Id$ denotes the identity matrix of appropriate size), then $q$ is non-negative if and only if the least eigenvalue of $A$ is greater than or equal to $-2$.

Loop-lees bigraphs whose symmetric adjacency matrix has least eigenvalue greater than or equal to $-2$ have attracted the attention of graph theorists since the 1960's~\cite{HN61,mA67,GH68,CGSS}, mainly focusing in their graphical characterization and Laplacian (or Kirchhoff) spectral properties. The connection of these bigraphs with the classical root systems $\A\D\E$ was established in the seminal work~\cite{CGSS} by Cameron, Goethals, Seidel and Shult in 1976 (see also~\cite{tZ81}).

The theory of (loop-lees) bigraphs with least eigenvalue $-2$, and the theory of quadratic (unit) forms, have maintained fairly independent roads (see for instance~\cite{dS20} and references therein), perhaps due to their seemingly different graphical and algebraic goals. The author translated in~\cite{jaJ2018} the inflation techniques into the combinatorial context of (a version of) line digraphs, using the so-called \emph{incidence quadratic form} of a quiver (directed multi-graph). Some of the basic concepts in~\cite{jaJ2018} had already been introduced by Zaslavsky in~\cite[\S 3]{tZ08}, see also~\cite{BS16,tZ84}.  Here we further explore this connection. More on the development of the theory of line graph can be found in~\cite{CDD21}.

Motivated by results of Barot~\cite{B99} and von H{\"o}hne~\cite{vH88}, the author associated to any quiver $Q$ a bigraph $\Inc(Q)$, called \emph{incidence graph} of $Q$~\cite{jaJ2018}. This construction has many similarities to the so-called \emph{line digraph} (introduced by Harary and Norman~\cite{HN61} in 1961), and is used as an auxiliary construction for \emph{signed line graphs} (see~\cite{tZ08,BS16}). The theory of flations for non-negative unit forms of Dynkin type $\A_n$, in this combinatorial context, was worked out in~\cite{jaJ2018}. Here we propose a modified theory of flations that preserves strong Gram congruence of unit forms (Section~\ref{S(C)}), given a solution of Problem~A for positive unit forms of Dynkin type $\A_n$:

\medskip\smallskip\noindent
\textbf{Theorem A.}
Any two connected positive unit forms of Dynkin type $\A_n$ are strongly Gram congruent.

\medskip\smallskip
An essentially combinatorial proof of Theorem~A is given below (Theorem~\ref{T:20}), after some technical preparations. An equivalent result was obtained recently by Simson in~\cite{dS21b}, in the context of edge-bipartite graphs and morsifications of quadratic forms. Recall that a positive unit form is positive precisely when it is non-negative and has corank zero (see Lemma~\ref{L:01}$(c)$ below). A connected non-negative unit form of corank one is called principal. In Section~\ref{S(A)} we present a combinatorial formula for the Coxeter matrix in terms of quiver inverses, whose similarity invariants (for instance, the characteristic polynomial, called Coxeter polynomial of the unit form), serve as discriminant in the analogous of Theorem~A for principal unit forms (Theorem~\ref{T:34}):

\medskip\smallskip\noindent
\textbf{Theorem B.}
Two connected principal unit forms of Dynkin type $\A_n$ are strongly Gram congruent if and only if they have the same Coxeter polynomial.

\medskip\smallskip
A list of the corresponding Coxeter polynomials is given in Remark~\ref{R:29medio}. Quiver inverses are used in Corollary~\ref{C:28} to give bounds for the coefficients of the Coxeter matrix associated to non-negative unit forms $q$ of Dynkin type $\A_n$. In Section~\ref{Basic} we collect concepts and general results needed throughout the paper.

\section{Basic notions} \label{Basic}

The set of integers is denoted by $\Z$, and the canonical basis of $\Z^n$ by $\bas_1,\ldots,\bas_n$. All matrices have integer coefficients, and for a $n \times m$ matrix $A$ and a $n \times m'$ matrix $B$, the $n \times (m + m')$ matrix with columns those of $A$ and $B$ (in this order), is denoted by $[A|B]$. In particular, if $A_1,\ldots,A_m$ are the columns of $A$, we write $A=[A_1|\cdots |A_m]$. The identity $n \times n$ matrix is denoted by $\Id_n$, and simply by $\Id$ for appropriate size. The transpose of a matrix $A$ is denoted by $A^{\tr}$, and if $A$ is an invertible square matrix, then $A^{-\tr}$ denotes $(A^{-1})^{\tr}$. By total (or linear) order we mean a partial order where any two elements are comparable, and a totally (or linearly) ordered set is one equipped with such order.

\subsection{Integral quadratic forms} \label{S(I):def}

Let $q=q(x_1,\ldots,x_n)$ be an \textbf{integral quadratic form} on $n \geq 1$ variables, that is, $q$ is a homogeneous polynomial of degree two,
\[
q(x_1,\ldots,x_n)=\sum_{1 \leq i\leq j \leq n}q_{ij}x_ix_j.
\]
In case $q_{ii}=1$ for $i=1,\ldots,n$ we say that $q$ is a \textbf{unit form}. The \textbf{(upper) triangular Gram matrix} associated to an integral quadratic form $q$ is the $n \times n$ matrix given by $\widecheck{G}_q=(g_{ij})$ where $g_{ij}=q_{ij}$ for $1 \leq i \leq j \leq n$ and $g_{ij}=0$ for $1 \leq j < i \leq n$. The \textbf{symmetric Gram matrix} associated to $q$ is given by $G_q=\widecheck{G}_q+\widecheck{G}_q^{\tr}$.

\medskip
As usual, the form $q$ can be seen as a function $q:\Z^n \to \Z$ given by evaluation in the vector of variables $(x_1,\ldots,x_n)$. Notice that for $x=(x_1,\ldots,x_n)^{\tr} \in \Z^n$ we have
\[
q(x)=x^{\tr}\widecheck{G}_qx=\frac{1}{2}x^{\tr}G_qx.
\]
For an endomorphism $T:\Z^n \to \Z^n$, the integral quadratic form $qT$ is given by $qT(x)=q(T(x))$. Observe that $G_{qT}=T^{\tr}G_qT$ (where here as in the rest of the text we identify an endomorphism $T$ with its square matrix under the ordered canonical basis $\bas_1,\ldots,\bas_n$ of $\Z^n$). We say that two unit forms $q$ and $q'$ are \textbf{weakly congruent} if there is an automorphism $T$ of $\Z^n$ such that $q'=qT$ (written $q' \sim q$ or $q' \sim^T q$).

\medskip
The \textbf{direct sum} $q \oplus q': \Z^{n+n'} \to \Z$ of integral quadratic forms $q:\Z^n\to \Z$ and $q':\Z^{n'} \to \Z$ is given by
\[
(q \oplus q')(x_1,\ldots,x_{n+n'})=q(x_1,\ldots,x_n)+q'(x_{n+1},\ldots,x_{n+n'}).
\]
Observe that $\widecheck{G}_{q\oplus q'}=\begin{pmatrix} \widecheck{G}_{q}&0\\0&\widecheck{G}_{q'} \end{pmatrix}$, which we denote by $\widecheck{G}_q\oplus \widecheck{G}_{q'}$. The \textbf{symmetric bilinear form} associated to an integral quadratic form $q:\Z^n \to \Z$, denoted by $q(-|-):\Z^n\times \Z^n \to \Z$, is given by
\[
q(x|y)=q(x+y)-q(x)-q(y),
\]
(notice that $q(x|y)=x^{\tr}G_qy$ for any $x,y \in \Z^n$). The \textbf{radical} $\rad(q)$ of a unit form $q$ is the set of vectors $x$ in $\Z^n$ such that $q(x|-)\equiv 0$ (called \textbf{radical vectors} of $q$). Clearly, $\rad(q)$ is a subgroup of $\Z^n$, whose rank $\CRnk(q)$ is called \textbf{corank} of $q$. Alternatively, $\CRnk(q)=n-\Rnk(q)$, where $\Rnk(q)=\Rnk(G_q)$ is called \textbf{rank} of $q$. By \textbf{root} of $q$ we mean a vector $x \in \Z^n$ such that $q(x)=1$.

\medskip
For convenience, throughout the text we use the notation $q_{ji}=q_{ij}$ for $i<j$.  A unit form $q$ is said to be \textbf{connected} if for any indices $i \neq j$ there is a sequence of indices $i_0,\ldots,i_r$ with $r \geq 1$ such that $i=i_0$, $j=i_r$ and $q_{i_{t-1}i_t}\neq 0$ for $t=1,\ldots,r$. Recall that a unit form $q:\Z^n \to \Z$ is called \textbf{non-negative} (resp. \textbf{positive}) if $q(x)\geq 0$ for any vector $x$ in $\Z^n$  (resp. $q(x)>0$ for any non-zero vector $x$ in $\Z^n$). A unit form $q$ is called \textbf{principal} if $q$ is non-negative and has corank one, see Simson~\cite{dS11} and Kosakowska~\cite{jK12}. The following important observation is well known (see for instance~\cite{BP99}), for convenience here we give a short proof.

\begin{lemma}\label{L:01}
For a unit form $q$ the following assertions hold.
\begin{itemize}
\itemsep=0.9pt
 \item[a)] If $q$ is not connected, then $q \sim q^1\oplus q^2$ for unit forms $q^1$ and $q^2$.
 \item[b)] If $q$ is non-negative and $q(x)=0$, then $x$ is a radical vector of $q$.
 \item[c)] The unit form $q$ is positive if and only if $q$ is non-negative and $\CRnk(q)=0$.
 \end{itemize}
Assume now that $q$ is a non-negative unit form, and that $q'$ is a unit form weakly congruent to $q$.
\begin{itemize}
\itemsep=0.9pt
 \item[d)] Then $q'$ is non-negative and $\CRnk(q')=\CRnk(q)$.
 \item[e)] If $q$ is connected, then so is $q'$.
\end{itemize}
\end{lemma}

\begin{proof}
A more specific version of $(a)$ will be given later in Lemma~\ref{L:05}. For $(b)$, take a basic vector $\bas_i$ and any integer $m$. If $q(x)=0$ for a vector $x$, then
\[
0 \leq q(mx+\bas_i)=mq(x|\bas_i)+1.
\]
Since $m$ and $i$ are arbitrary, then $q(x|\bas_i)=0$ and $x$ is a radical vector of $q$.
\eject

For $(c)$, if $q$ is positive then clearly $q$ is non-negative and $\rad(q)=0$. Conversely, if $q$ is non-negative and $\CRnk(q)=0$, then $\rad(q)=0$. By $(b)$, for any non-zero vector $x$ we have $q(x)>0$. To prove $(d)$ consider an automorphism $T$ such that $q'=qT$. Then $q'(x)=q(T(x))\geq 0$, that is, $q'$ is non-negative. Since
\[
G_{q'}=T^{\tr}G_qT,
\]
it is well known that $\Rnk(G_{q'})=\Rnk(G_q)$ (cf.~\cite[\S~4.5]{cdM00}), and therefore $\CRnk(q')=\CRnk(q)$.

\medskip
Finally, to show $(e)$ assume that $q'=qT$ is not connected. By $(a)$ we may assume that $q'=q^1 \oplus q^2$ for unit forms $q^1:\Z^{n_1} \to \Z$ and $q^2:\Z^{n_2} \to \Z$ (with $n=n_1+n_2$). Let $y_1,\ldots,y_n$ be the columns of $T^{-1}$. Since $q$ is a unit form, $y_i$ is a root of $q'$ for $i=1\ldots,n$. Moreover, if $y_i^1$ (resp. $y_i^2$) is the projection of $y_i$ into its first $n_1$ entries (resp. its last $n_2$ entries), then
\[
1=q'(y_i)=q^1(y^1_i)+q^2(y^2_i).
\]
By $(d)$, the unit forms $q^1$ and $q^2$ are non-negative, therefore either $q^1(y^1_i)=1$ and $q^2(y^2_i)=0$, or $q^1(y^1_i)=0$ and $q^2(y^2_i)=1$. Consider the following partition of the set $\{1,\ldots,n\}$,
\[
X=\{1 \leq i \leq n \mid q^1(y^1_i)=1\} \quad \text{and} \quad Y=\{1 \leq j \leq n \mid q^2(y^2_j)=1\}.
\]
By $(b)$, observe that if $i \in X$ and $j \in Y$, then $y_j^1$ is a radical vector of $q^1$ and $y_i^2$ is a radical vector of $q^2$, and therefore
\[
q(\bas_i+\bas_j)=q'(y_i+y_j)=q^1(y^1_i+y^1_j)+q^2(y^2_i+y^2_j)=q^1(y^1_i)+q^2(y^2_j)=2.
\]
Then $q(\bas_i+\bas_j)=q(\bas_i)+q(\bas_j)$, which implies that $q_{ij}=0$ for arbitrary $i \in X$ and $j \in Y$ (for $q_{ij}=q(\bas_i|\bas_j)$). We need to show that both $X$ and $Y$ are non-empty sets.

\medskip
We may write
\[
(T^{-1})^{\tr}G_{q'}T^{-1}=\begin{pmatrix} B^{\tr}&C^{\tr} \end{pmatrix}\begin{pmatrix} G_{q^1}&0\\0&G_{q^2} \end{pmatrix}\begin{pmatrix} B\\C \end{pmatrix}=B^{\tr}G_{q^1}B+C^{\tr}G_{q^2}C,
\]
where $B$ and $C$ are respectively $n_1 \times n$ and $n_2 \times n$ matrices. If $Y$ is an empty set, then the columns of $C$ are radical vectors of $q^2$, and therefore
\[
G_q=(T^{-1})^{\tr}G_{q'}T^{-1}=B^{\tr}G_{q^1}B.
\]
In particular, $\Rnk(q)\leq \Rnk(q^1)$ (cf.~\cite[\S 4.5]{cdM00}). This is impossible, since by $(d)$ we have $\Rnk(q)=\Rnk(q')=\Rnk(q^1)+\Rnk(q^2)>\Rnk(q^1)$ (a similar contradiction can be found assuming that $X$ is an empty set). This completes the proof, since $X \neq \emptyset$ and $Y \neq \emptyset$ imply that $q$ is non-connected.
\end{proof}

The example following Remark~\ref{R:02} below shows that the non-negativity assumption is necessary for part $(e)$ in Lemma~\ref{L:01}.

\subsection{Bigraphs and associated unit forms} \label{S(I):bigraph}

Let $\Gamma=(\Gamma_0,\Gamma_1,\sigma)$ be a \textbf{bigraph}, that is, a multi-graph $(\Gamma_0,\Gamma_1)$ together with a \textbf{sign function} $\sigma:\Gamma_1 \to \{\pm 1\}$ such that all parallel edges have the same sign (see~\cite{dS13}). As usual, bigraphs are graphically depicted in the following way: for vertices $i$ and $j$, an edge $a$ joining $i$ and $j$ with $\sigma(a)=1$ will be denoted by $\xymatrix{{\bullet_i} \ar@{-}[r] & {\bullet_j}}$, and by $\xymatrix{{\bullet_i} \ar@{.}[r] & {\bullet_j}}$ if $\sigma(a)=-1$ (this convention is used in~\cite{BP99} and~\cite{jaJ2018}, and is opposite to the one used in~\cite{dS13}). We assume that the set of vertices $\Gamma_0$ is totally ordered.  If $\Gamma$ has no loop and $|\Gamma_0|=n$, the \textbf{(upper) triangular adjacency matrix} $\Adj_{\Gamma}$ of $\Gamma$ is the $n \times n$ matrix given by
\[
\Adj_{\Gamma}=\begin{pmatrix} 0 & d_{1,2} & d_{1,3} & \cdots & d_{1,n-1} & d_{1,n} \\
0 & 0 & d_{2,3} & \cdots & d_{2,n-1} & d_{2,n} \\ 0 & 0 & 0 & \vdots & d_{3,n-1} & d_{3,n} \\ \vdots & \vdots & \vdots & \ddots & \vdots & \vdots \\
0 & 0 & 0 & \cdots & 0 & d_{n-1,n} \\
0 & 0 & 0 & \cdots & 0 & 0 \end{pmatrix},
\]
where $|d_{ij}|$ is the number of edges between vertices $i<j$, and $\sigma(a)d_{ij}=|d_{ij}|$ for any such edge $a$. Throughout the text we assume that all bigraphs are \textbf{loop-less}, that is, have no loop.

\medskip
The \textbf{(upper) triangular Gram matrix} $\widecheck{G}_{\Gamma}$ of $\Gamma$ is given by $\widecheck{G}_{\Gamma}=\Id-\Adj_{\Gamma}$ (following the convention in~\cite{dS13} one gets $\widecheck{G}_{\Gamma}=\Id+\Adj_{\Gamma}$). The \textbf{quadratic form associated to a bigraph $\Gamma$} is given by
\[
q_{\Gamma}(x)=x^{\tr}\widecheck{G}_{\Gamma}x, \quad \text{for any $x \in \Z^n$,}
\]
that is, $\widecheck{G}_{q_{\Gamma}}=\widecheck{G}_{\Gamma}$. There is a well known bijection between unit forms and loop-less bigraphs (see for instance~\cite{dS13} and~\cite{jaJ2018}), given by the corresponding triangular (Gram and adjacency) matrices.

\begin{remark}\label{R:02}
For a loop-less bigraph $\Gamma$, the unit form $q_{\Gamma}$ is connected if and only if $\Gamma$ is a connected bigraph.
\end{remark}

\begin{proof}
Take $q_{\Gamma}(x)=\sum \limits_{1 \leq i \leq j \leq n}q_{ij}x_ix_j$, and observe that for any sequence of indices $i_0,\ldots,i_r$ such that $q_{i_{t-1}i_t}\neq 0$ for $t=1,\ldots,r$, there is a sequence of edges $a_1,\ldots,a_r$ in $\Gamma$ such that $a_t$ contains vertices $i_{t-1}$ and $i_t$ (that is, $(i_0,a_1,i_1,\ldots,i_{r-1},a_r,i_r)$ is a walk in $\Gamma$). Hence the claim follows.
\end{proof}

Consider the following connected bigraph $\Gamma$ with twelve edges, and non-connected bigraph $\Gamma'$ with six edges,
 \[
\Gamma= 
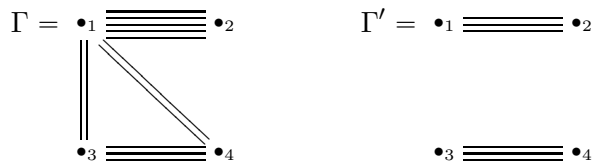
\xymatrix@R=1.3pc@C=1.3pc{\mathmiddlescript{\bullet_1} \ar@{-}@<-.5ex>[dd] \ar@{-}[dd] \ar@{-}@<-.5ex>[rrdd] \ar@{-}[rrdd] \ar@{-}[rr] \ar@{-}@<.5ex>[rr] \ar@{-}@<-.5ex>[rr] \ar@{-}@<1ex>[rr] \ar@{-}@<-1ex>[rr] & &  \mathmiddlescript{\bullet_2} \\ \\
 \mathmiddlescript{\bullet_3} \ar@{-}[rr] \ar@{-}@<-.5ex>[rr] \ar@{-}@<.5ex>[rr] & & \mathmiddlescript{\bullet_4}}
 \qquad \qquad \Gamma'= \xymatrix@R=1.3pc@C=1.3pc{\mathmiddlescript{\bullet_1} \ar@{-}[rr] \ar@{-}@<.5ex>[rr] \ar@{-}@<-.5ex>[rr]  & &  \mathmiddlescript{\bullet_2} \\ \\
 \mathmiddlescript{\bullet_3} \ar@{-}[rr] \ar@{-}@<-.5ex>[rr] \ar@{-}@<.5ex>[rr] & & \mathmiddlescript{\bullet_4}}
 \]
None of the unit forms $q_{\Gamma}$ and $q_{\Gamma'}$ is non-negative, and $q_{\Gamma'} \sim q_{\Gamma}$. Indeed, we have
\[
G_{\Gamma}=\left(\begin{smallmatrix} 2&\mi{5}&\mi{2}&\mi{2}\\\mi{5}&2&0&0\\\mi{2}&0&2&\mi{3}\\\mi{2}&0&\mi{3}&2 \end{smallmatrix}\right) \quad \text{and} \quad G_{\Gamma'}=\left(\begin{smallmatrix}2&\mi{3}&0&0\\\mi{3}&2&0&0\\0&0&2&\mi{3}\\0&0&\mi{3}&2 \end{smallmatrix}\right),
\]
where for an integer $a$ we take $\mi{a}=-a$. A direct calculation shows that if $B=\left(\begin{smallmatrix} 1&0&0&0\\1&1&0&0\\\mi{2}&0&1&0\\\mi{2}&0&0&1 \end{smallmatrix}\right)$, then $G_{\Gamma'}=B^{\tr}G_{\Gamma}B$.

\subsection{Elementary transformations for unit forms} \label{S(I):trans}

We consider the following transformations of a unit form $q:\Z^n \to \Z$.

\begin{enumerate}
 \item \textit{Point inversion.} Take a subset of indices $C \subseteq \{1,\ldots,n\}$ and define the automorphism $V_C:\Z^n \to \Z^n$ given by $V_C(\bas_k)=-\bas_k$ if $k \in C$, and $V_C(\bas_k)=\bas_k$ otherwise. The transformation $V_C$ is known as \textbf{point inversion} (or sign change) for $q$, and the unit form $qV_C$ is usually referred to as \textbf{point inversion} of $q$.

 \item \textit{Swapping.} Given two indices $i \neq j$, consider the transformation $S_{ij}:\Z^n \to \Z^n$ given by $S_{ij}(\bas_i)=\bas_j$, $S_{ij}(\bas_j)=\bas_i$ and $S_{ij}(\bas_k)=\bas_k$ for $k \neq i,j$ (clearly, $S_{ij}=S_{ji}$). We say that the unit form $qS_{ij}$ is obtained from $q$ by \textbf{swapping} indices $i$ and $j$.

 \item \textit{Flation.} For two indices $i \neq j$, consider the sign $\epsilon=\sgn(q_{ij})\in \{+1,0,-1\}$ of $q_{ij}$. Take the linear transformation $T^{\epsilon}_{ij}:\Z^n \to \Z^n$ given by $T^{\epsilon}_{ij}(x)=x-\epsilon x_i\bas_j$, for a (column) vector $x=(x_1,\ldots,x_n)^{\tr}$ in $\Z^n$. The transformation $T^{\epsilon}_{ij}$ will be referred to as \textbf{flation} for $q$.

 \item \textit{$FS$-transformation.} For our arguments we consider the composition $\FSq^{\epsilon}_{ij}=T^{\epsilon}_{ij}S_{ij}$, and call it a \textbf{$FS$-(linear) transformation} for $q$ if $\epsilon=\sgn(q_{ij})$.
\end{enumerate}

\begin{remark}\label{R:03}
Let $q$ be a unit form, with indices $i \neq j$.
\begin{itemize}
\item[a)] Let $T^{\epsilon}_{ij}$ be a flation for $q$. Then $q'=qT^{\epsilon}_{ij}$ is a unit form if and only if $|q_{ij}|\leq 1$, and in that case $q'T^{-\epsilon}_{ij}=q$.

\item[b)] Let $\FSq^{\epsilon}_{ij}$ be a $FS$-transformation for $q$. Then $q'=q\FSq^{\epsilon}_{ij}$ is a unit form if and only if $|q_{ij}|\leq 1$, and in that case $q'\FSq^{-\epsilon}_{ji}=q$.
\end{itemize}
\end{remark}

\begin{proof}
To show $(a)$, observe that $q'(\bas_k)=q(\bas_k)=1$ if $k \neq i$, and
\[
q'(\bas_i)=q(T^{\epsilon}_{ij}(\bas_i))=q(\bas_i-\epsilon\bas_j)=1+(\epsilon-|q_{ij}|).
\]
Since $\epsilon=\sgn(q_{ij})$, it follows that $q'(\bas_i)=1$ if and $|q_{ij}|\leq 1$. That $T_{ij}^{\epsilon}T_{ij}^{-\epsilon}=\Id$ is clear, since
\[
T_{ij}^{\epsilon}T_{ij}^{-\epsilon}(x)=T_{ij}^{\epsilon}(x+\epsilon x_i\bas_j)=T_{ij}^{\epsilon}(x)+\epsilon x_iT_{ij}^{\epsilon}(\bas_j)=x-\epsilon x_i\bas_j+\epsilon x_i\bas_j=x,
\]
for any $x$ in $\Z^n$. Claim $(b)$ follows from $(a)$.
\end{proof}

A composition $T=T_{i_1j_1}^{\epsilon_1}\cdots T_{i_rj_r}^{\epsilon_r}$ is called an \textbf{iterated flation} for a unit form $q$, if taking $q^0=q$, then $T_{i_tj_t}^{\epsilon_t}$ is a flation for $q^{t-1}$, and $q^t=q^{t-1}T_{i_tj_t}^{\epsilon_t}$ is a unit form, for $t=1,\ldots,r$. In a similar situation, a composition $\FSq=\FSq_{i_1j_1}^{\epsilon_1}\cdots \FSq_{i_rj_r}^{\epsilon_r}$ is called an \textbf{iterated $FS$-transformation} for $q$.

\subsection{Weak classification and strong congruence} \label{S(I):cong}

The classification of positive unit forms up to weak congruence is classical: for $n  \geq 1$, the weak equivalence classes of connected positive unit forms in $n$ variables are in correspondence with the set of (simply laced) Dynkin types $\mathcal{D}_n$, where
\begin{equation*}
\mathcal{D}_n = \left\{ \begin{array}{l l} \{\A_n\}, & \text{if $n=1,2,3$},\\ \{\A_n,\D_n\}, & \text{if $n=4,5$ or $n \geq 9$},\\ \{\A_n,\D_n,\E_n\}, & \text{if $n=6,7,8$},\end{array} \right.
\end{equation*}
(called \textbf{Dynkin graphs} on $n$ vertices, see Table~\ref{T:dynkin}). The following weak classification of connected non-negative unit forms can be found in~\cite{SZ17} (see also~\cite{BP99}). If $J \subset \{1,\ldots,n\}$ is a subset of indices, denote by $\tau:\Z^J \to \Z^n$ the canonical inclusion. If $q:\Z^n \to \Z$ is a unit form, then the composition $q'=q\tau$ is a unit form called \textbf{restriction} of $q$, and the unit form $q$ is called \textbf{extension} of $q'$. Simson fixed in~\cite{dS16a} (see also~\cite[Algorithm~3.18]{SZ17}) canonical extensions of connected positive unit forms, via corresponding connected bigraphs $\widehat{\Delta}^{(c)}$ for $c \geq 0$ and $\Delta$ a Dynkin graph, which serve as representatives of weak congruence classes of connected positive unit forms.

\begin{table}[!h]
\begin{center}
\caption{(Simply laced) Dynkin graphs with ordered set of vertices.} \label{T:dynkin}
\renewcommand{\arraystretch}{1.3}
\begin{tabular}{c l}
\hline  Notation & \multicolumn{1}{c}{Graph} \\
\hline \\ [-5pt]
$\A_n \; (n \geq 1)$ & $\xymatrix@C=1pc@R=1pc{\mathmiddlescript{\bullet_1} \ar@{-}[r] & \mathmiddlescript{\bullet_2} \ar@{-}[r] & \mathmiddlescript{\bullet_3} \ldots \mathmiddlescript{\bullet_{n-2}} \ar@{-}[r] & \mathmiddlescript{\bullet_{n-1}} \ar@{-}[r] & \mathmiddlescript{\bullet_n}} $ \\\raisebox{-1ex}{$\D_n \; (n \geq 4)$} & $\xymatrix@C=1pc@R=.2pc{\mathmiddlescript{\bullet_1} \ar@{-}[rd] \\  & \mathmiddlescript{\bullet_3} \ar@{-}[r] & \mathmiddlescript{\bullet_4} \ldots \mathmiddlescript{\bullet_{n-2}} \ar@{-}[r] & \mathmiddlescript{\bullet_{n-1}} \ar@{-}[r] & \mathmiddlescript{\bullet_n} \\ \mathmiddlescript{\bullet_2} \ar@{-}[ru] }  $ \\ \raisebox{-2ex}{$\E_6$} & $\xymatrix@C=1pc@R=1pc{ & & \mathmiddlescript{\bullet_6} \ar@{-}[d] \\ \mathmiddlescript{\bullet_1} \ar@{-}[r] & \mathmiddlescript{\bullet_2} \ar@{-}[r] & \mathmiddlescript{\bullet_3} \ar@{-}[r] & \mathmiddlescript{\bullet_4} \ar@{-}[r] & \mathmiddlescript{\bullet_5}} $ \\ \raisebox{-2ex}{$\E_7$} & $\xymatrix@C=1pc@R=1pc{ & & \mathmiddlescript{\bullet_4} \ar@{-}[d] \\  \mathmiddlescript{\bullet_1} \ar@{-}[r] & \mathmiddlescript{\bullet_2} \ar@{-}[r] & \mathmiddlescript{\bullet_3} \ar@{-}[r] & \mathmiddlescript{\bullet_5} \ar@{-}[r] & \mathmiddlescript{\bullet_6} \ar@{-}[r] & \mathmiddlescript{\bullet_7}} $ \\ \raisebox{-2ex}{$\E_8$} & $\xymatrix@C=1pc@R=1pc{ & & \mathmiddlescript{\bullet_4} \ar@{-}[d] \\ \mathmiddlescript{\bullet_1} \ar@{-}[r] & \mathmiddlescript{\bullet_2} \ar@{-}[r] & \mathmiddlescript{\bullet_3} \ar@{-}[r] & \mathmiddlescript{\bullet_5} \ar@{-}[r] & \mathmiddlescript{\bullet_6} \ar@{-}[r] & \mathmiddlescript{\bullet_7} \ar@{-}[r] & \mathmiddlescript{\bullet_8}} $ \\
 \multicolumn{1}{c}{}\\[-5pt]
\hline
\end{tabular}
\end{center}
\end{table}

\begin{theorem}[Simson-Zaj\k{a}c, 2017]\label{T:04}
Let $q:\Z^n \to \Z$ be a non-negative connected unit form of corank $\CRnk(q)=c$. Then there exists an iterated flation $B:\Z^n \to \Z^n$ and a unique Dynkin graph $\Delta \in \mathcal{D}_{n-c}$, denoted by $\Dyn(q)=\Delta$ and called \textbf{Dynkin type} of $q$, such that the unit form $qB$ is the canonical $c$-extension of $q_{\Delta}$. In particular, two non-negative unit forms $q$ and $q'$ are weakly congruent if and only if they are of the same corank and of the same Dynkin type.
\end{theorem}

Two unit forms $q$ and $q'$ are said to be \textbf{strongly (Gram) congruent}, written $q' \approx q$ or $q' \approx^B q$, if there is an automorphism $B$ such that
\[
\widecheck{G}_{q'}=B^{\tr}\widecheck{G}_qB.
\]

A complete classification of strong congruence classes for the exceptional Dynkin types $\E_6$, $\E_7$ and $\E_8$ is given in~\cite{dS20,dS18}, as well as for the non-simply laced Dynkin types $\B_n$, $\C_n$, $\F_4$ and $\G_2$. Similar results for the Dynkin type $\D_n$ were announced in~\cite{dS18} and proved  in~\cite{dS21a}.

\begin{lemma}\label{L:05}
Let $q:\Z^n \to \Z$ be a unit form. If $q$ is not connected, then $q \approx q^1 \oplus q^2$ for unit forms $q^1$ and $q^2$, and $q^2 \oplus q^1 \approx q^1 \oplus q^2$. In particular, strong congruence preserves connectedness of non-negative unit forms.
\end{lemma}

\begin{proof}
Assume that $q$ is not connected, and take non-empty sets $X$ and $Y$ such that $q_{ij}=0$ for $i \in X$ and $j \in Y$. There is a permutation $\rho$ of the set $\{1,\ldots,n\}$ satisfying
\begin{itemize}
 \item If $i<j$ and $i,j \in X$ or $i,j \in Y$, then $\rho(i)<\rho(j)$.
 \item If $i \in X$ and $j \in Y$, then $\rho(i)<\rho(j)$.
\end{itemize}
Let $P=[\bas_{\rho^{-1}(1)}|\cdots |\bas_{\rho^{-1}(n)}]$ be the permutation matrix associated to the inverse permutation $\rho^{-1}$, and consider the product
\[
A=(a_{ij})_{i=1}^n=P^{\tr}\widecheck{G}_qP.
\]
Then $a_{\rho(i)\rho(j)}=q_{ij}$, and by the conditions on the permutation $\rho$, the matrix $A$ has the following block-diagonal shape,
\[
A=\begin{pmatrix} A_1&0\\0&A_2 \end{pmatrix}.
\]
Moreover, $A_1$ and $A_2$ are upper diagonal matrices with ones in their main diagonals. Taking $q^i$ such that $\widecheck{G}_{q^i}=A_i$ (for $i=1,2$), we get
\[
P^{\tr}\widecheck{G}_qP=\widecheck{G}_{q^1}\oplus \widecheck{G}_{q^2}=\widecheck{G}_{q^1 \oplus q^2},
\]
as wanted. Swapping the sets $X$ and $Y$ we get $q \approx q^2 \oplus q^1$.

\medskip
The last claim follows by Lemma~\ref{L:01}$(e)$, since strong congruence implies weak congruence.
\end{proof}

\subsection{Elementary quiver transformations} \label{S(I):Quiv}

A \textbf{quiver} $Q=(Q_0,Q_1,\sou,\tar)$ consists of (finite) sets $Q_0$ and $Q_1$, whose elements are called \textbf{vertices} and \textbf{arrows} of $Q$ respectively, and functions $\sou,\tar:Q_1 \to Q_0$. We say that the vertices $v$ and $w$ are \textbf{source} and \textbf{target} of an arrow $i$ respectively, if $\sou(i)=v$ and $\tar(i)=w$, and display $i$ graphically as $\xymatrix{v \ar[r]^-{i} & w}$ if $v \neq w$. For convenience we consider the set $\MiNu(i)=\{\sou(i),\tar(i)\}$ of vertices \textbf{incident} to $i$.  An arrow $i$ with $\sou(i)=\tar(i)$ is called a \textbf{loop} of $Q$, and we say that $Q$ is a \textbf{loop-less quiver} if it contains no loop. Two arrows $i$ and $j$ are said to be \textbf{parallel} if $\MiNu(i) = \MiNu(j)$, and \textbf{adjacent} if $\MiNu(i)\cap \MiNu(j)\neq \emptyset$. The \textbf{degree} of a vertex $v$ in $Q$ is the number of arrows $i$ in $Q$ such that $v \in \MiNu(i)$. A vertex in $Q$ with degree one is called a \textbf{leaf}, and any arrow $i$ such that $\MiNu(i)$ contains a leaf is called a \textbf{pendant arrow} of $Q$. Observe that $\overline{Q}=(Q_0,\{\MiNu(i)\}_{i \in Q_1})$ is a multi-graph, referred to as \textbf{underlying graph} of $Q$. We say that $Q$ is \textbf{simple}, \textbf{connected} or a \textbf{tree} if so is $\overline{Q}$, and walks of $\overline{Q}$ are also called walks of $Q$. To be precise and fix some notation, by \textbf{walk} of $Q$ we mean an alternating sequence of vertices and arrows in $Q$ of the form,
\[
\alpha=(v_0,i_1,v_1,\ldots,v_{\ell-1},i_{\ell},v_{\ell}),
\]
such that $\MiNu(i_t)=\{v_{t-1},v_t\}$ for $t=1,\ldots,\ell$. The notation $\alpha=i_1^{\epsilon_1}\cdots i_{\ell}^{\epsilon_{\ell}}$ for $\epsilon_t=\pm 1$ is also used, where the symbol $i_t^{+1}$ stands for $\sou(i_t)=v_{t-1}$ and $\tar(i_t)=v_t$, while $i_t^{-1}$ stands for $\sou(i_t)=v_{t}$ and $\tar(i_t)=v_{t-1}$ (that is, the sign $\epsilon_t$ denotes the direction in which the arrow $i_t$ is found along the walk $\alpha$). The non-negative integer $\ell$ is called \textbf{length} of the walk $\alpha$, and we call vertex $v_0$ (resp. vertex $v_{\ell}$) the \textbf{starting vertex} (resp. the \textbf{ending} vertex) of $\alpha$. A walk with length zero is called a \textbf{trivial walk}. As usual we omit the exponent $+1$ in our notation of walks, and abusing notation we set $\sou(\alpha)=v_0$ and $\tar(\alpha)=v_{\ell}$. Observe that the reversed sequence
\[
(v_{\ell},i_{\ell},v_{\ell-1},\ldots,v_1,i_1,v_0),
\]
is also a walk in $Q$, referred to as \textbf{reverse walk} of $\alpha$ and denoted by $\alpha^{-1}$.  With our notation we have
\[
(i_1^{\epsilon_1}\cdots i_{\ell}^{\epsilon_{\ell}})^{-1}=i_{\ell}^{-\epsilon_{\ell}}\cdots i_1^{-\epsilon_1},
\]
(in particular $\sou(i^{-1})=\tar(t)$ and $\tar(i^{-1})=\sou(i)$).

\medskip
Let $Q=(Q_1,Q_0)$ be a quiver (we will usually exclude the source and target functions $\sou$ and $\tar$ from the notation of $Q$). For a vertex $v \in Q_0$, we denote by $Q^{(v)}$ the quiver obtained from $Q$ by removing the vertex $v$, as well as all arrows containing it (that is, all arrows $i$ with $v \in \MiNu(i)$). Similarly, if $i \in Q_1$ is an arrow of $Q$, we denote by $Q^{(i)}$ the quiver obtained from $Q$ by removing the arrow $i$.

\medskip
We will need the following transformations of a loop-less quiver $Q=(Q_0,Q_1)$. Throughout the text we assume that both sets $Q_0$ and $Q_1$ are totally ordered.
\begin{enumerate}
 \item \label{LbOne} \textit{Arrow inversion.} Let $C$ be a set of arrows in $Q$, and take $Q\mathcal{V}_C$ to be the quiver obtained from $Q$ by inverting the direction of all arrows in $C$.

 \item \textit{Swapping.} Given two arrows $i \neq j$ in $Q$, define the new quiver $Q\mathcal{S}_{ij}$ as the quiver obtained from $Q$ by \textbf{swapping} arrows $i$ and $j$ (therefore, swapping their positions in the total ordering of $Q_1$).

 \item \textit{(Quiver) Flation.}  Let $i$ and $j$ be adjacent arrows in $Q$, and choose signs $\epsilon_i$ and $\epsilon_j$ such that $\alpha=i^{\epsilon_i}j^{\epsilon_j}$ is a walk in $Q$. Consider the quiver $Q'$ obtained from $Q$ by replacing $i$ by a new arrow $i'$ having $\sou(i')=\sou(\alpha)$ and $\tar(i')=\tar(\alpha)$ if $\epsilon_i=1$, and $\sou(i')=\tar(\alpha)$ and $\tar(i')=\sou(\alpha)$ if $\epsilon_i=-1$. The new arrow $i'$ takes the place of the deleted arrow $i$ in the ordering of $Q_1$. We will use the notation $Q'=Q\mathcal{T}_{ij}^{\epsilon}$ where $\epsilon=-\epsilon_i\epsilon_j$, and say that $\mathcal{T}_{ij}^{\epsilon}$ is a \textbf{flation} for the quiver $Q$. If $i$ and $j$ are non-adjacent arrows, we take $Q\mathcal{T}_{ij}^{0}=Q$.

 \item \textit{$FS$-transformation.} By \textbf{$FS$-transformation} of a quiver $Q$ with respect to the ordered pair of different arrows $(i,j)$, we mean the new quiver $Q\FS^{\epsilon}_{ij}$ given by
\[
Q\FS^{\epsilon}_{ij}=(Q\mathcal{T}^{\epsilon}_{ij})\mathcal{S}_{ij}.
\]
\end{enumerate}

The analogous of Remark~\ref{R:03} can be stated as follows.

\begin{remark}\label{R:06}
Let $Q$ be a loop-less quiver.
\begin{itemize}
\item[a)] Let $\mathcal{T}^{\epsilon}_{ij}$ be a flation for $Q$. Then $Q'=Q\mathcal{T}^{\epsilon}_{ij}$ is a loop-less quiver if and only if $i$ and $j$ are non-parallel arrows, and in that case $Q'\mathcal{T}^{-\epsilon}_{ij}=Q$.
\item[b)] Let $\FS^{\epsilon}_{ij}$ be a $FS$-transformation for $Q$. Then $Q'=Q\FS^{\epsilon}_{ij}$ is a loop-less quiver if and only if $i$ and $j$ are non-parallel arrows, and in that case $Q'\FS^{-\epsilon}_{ji}=Q$.
\end{itemize}
Moreover, the four types of transformations described above preserve the number of vertices and arrows of a quiver, as well as its connectedness if $i$ and $j$ are non-parallel arrows.
\end{remark}
\begin{proof}
Let $i$ and $j$ be non-parallel arrows in $Q$, and notice that by construction the corresponding arrows $i'$ and $j'$ in $Q'=Q\mathcal{T}^{\epsilon}_{ij}$ are non-parallel (for $i'$ takes the place of the walk $i^{\epsilon_i}j^{\epsilon_j}$ of $Q$, and $j'$ remains unchanged). Moreover, $i$ and $j$ are adjacent in $Q$ if and only if $i'$ and $j'$ are adjacent in $Q'$. Noticing that now $(i')^{\epsilon_i}(j')^{-\epsilon_j}$ is a walk in $Q'$, then $\mathcal{T}_{i'j'}^{-\epsilon}$ is a flation for $Q'$, and a second application of the construction above shows that $(Q')\mathcal{T}^{-\epsilon}_{i'j'}=Q$ (subsequently labels $i'$ and $j'$ are replaced by $i$ and $j$).  In particular,
\[
(Q\FS^{\epsilon}_{ij})\FS^{-\epsilon}_{ji}=([Q'S_{ij}]\mathcal{T}^{-\epsilon}_{ji})S_{ij}=([Q'\mathcal{T}^{-\epsilon}_{ij}]\mathcal{S}_{ij})S_{ij}=Q,
\]
since clearly $[Q'S_{ij}]\mathcal{T}^{-\epsilon}_{ji}=[Q'\mathcal{T}^{-\epsilon}_{ij}]\mathcal{S}_{ij}$. This shows $(a)$ and $(b)$.

\medskip
Now, if $Q'$ is the quiver obtained from $Q$ after applying any of the four types of transformations above, then $Q'_0=Q_0$, and $|Q'_1|=|Q_1|$. The claim on connectedness is clear for arrow inversions and swappings, and follows for $\mathcal{T}^{\epsilon}_{ij}$ and $\FS^{\epsilon}_{ij}$ using the arguments above.
\end{proof}

\section{Quivers and their incidence bigraphs} \label{S(C)}

The techniques presented in this section were introduced in a slightly wider context in~\cite{jaJ2018}, following ideas from Barot~\cite{B99} and von H{\"o}hne~\cite{vH88}.

\subsection{Incidence quadratic forms} \label{S(C):Iqf}

Consider an arbitrary quiver $Q$ with $|Q_0|=m$ vertices and $|Q_1|=n$ arrows, both sets $Q_0$ and $Q_1$ with fixed total orderings. The \textbf{(vertex-arrow) incidence matrix of $Q$} is the $m \times n$ matrix $I(Q)$ with columns $I_i=\bas_{\sou(i)}-\bas_{\tar(i)}$ for $i \in Q_1$ (observe that $I_i=0$ if and only if $i$ is a loop in $Q$), cf.~\cite{tZ08}.  For a loop-less quiver $Q$, it will be useful to consider the \textbf{incidence function} $\sigma_Q:Q_0\times Q_1 \to \{+1,0,-1\}$ given by
\begin{equation*}
\sigma_Q(v,i) = \left\{
\begin{array}{l l}
+1, & \text{if $v$ is the source of arrow $i$},\\
-1, & \text{if $v$ is the target of arrow $i$},\\
0, & \text{otherwise}.
\end{array} \right.
\end{equation*}
Clearly, $I(Q)=[\sigma_Q(v,i)]_{v \in Q_0}^{i \in Q_1}$. For a non-trivial walk $\alpha=i_1^{\epsilon_1}\cdots_{\ell}^{\epsilon_{\ell}}$ in $Q$, we use the notation $I_{\alpha}=\sum_{t=1}^{\ell}\epsilon_tI_{i_t}$. Observe that $I_{\alpha}$ is the telescopic sum
\[
I_{\alpha}=\epsilon_1(\bas_{\sou(i_1)}-\bas_{\tar(i_1)})+\ldots+\epsilon_{\ell}(\bas_{\sou(i_{\ell})}-\bas_{\tar(i_{\ell})})=\bas_{\sou(\alpha)}-\bas_{\tar(\alpha)}.
\]

\begin{definition}\label{D:new}
Let $Q$ be a quiver with $m \geq 1$ vertices and $n \geq 0$ arrows.
\begin{itemize}
 \item[a)] The square matrix $G_Q=I(Q)^{\tr}I(Q)$ is defined to be the \textbf{symmetric Gram matrix of $Q$}.
 \item[b)] Let $\widecheck{G}_Q$ be the (unique) upper triangular matrix such that $G_Q=\widecheck{G}_Q+\widecheck{G}_Q^{\tr}$, called \textbf{(upper) triangular Gram matrix} of $Q$.
 \item[c)] The quadratic form $q_Q:\Z^n \to \Z$ given by
\[
q_{Q}(x)=\frac{1}{2}x^{\tr}I(Q)^{\tr}I(Q)x=\frac{1}{2}||I(Q)x||^2,
\]
is defined to be the \textbf{quadratic form associated to $Q$}.
\end{itemize}
\end{definition}

Notice that the matrix $\widecheck{G}_Q$ is the standard matrix morsification of $q_Q$ in the sense of Simson~\cite{dS13a}.

\begin{remark}\label{R:new}
Observe that $G_Q$ and $\widecheck{G}_Q$ do not depend on the order given to the set of vertices $Q_0$. Indeed, if $Q'$ is a copy of $Q$ with different ordering of its vertices, there is a permutation $m \times m$ matrix $P$ such that $I(Q')=PI(Q)$, and therefore $I(Q')^{\tr}I(Q')=I(Q)^{\tr}P^{\tr}PI(Q)=I(Q)^{\tr}I(Q)$.
\end{remark}

To present a graphical description of $q_Q$, we need the following definition.

\begin{definition}\label{D:07}
Let $Q$ be a loop-less quiver with $m \geq 1$ vertices and $n \geq 0$ arrows. The \textbf{incidence (bi)graph $\Inc(Q)$} of a loop-less quiver $Q$ is defined as follows. The set of vertices $\Inc(Q)_0$ of $\Inc(Q)$ is the set of arrows of $Q$ (that is, $\Inc(Q)_0=Q_1$). The signed edges in $\Inc(Q)$ are given as follows:
\begin{itemize}
\itemsep=0.95pt
 \item[a)] Let $i$ and $j$ be adjacent non-parallel arrows in $Q$. Then the vertices $i$ and $j$ are connected by exactly one edge $a$ in $\Inc(Q)$, with sign $\sigma(a)$ given by $\sigma(a)=(-1)\sigma_Q(v,i)\sigma_Q(v,j)$ where $v$ is the (unique) vertex in $Q$ with $\{v\}=\MiNu(i)\cap \MiNu(j)$.

 \item[b)] Let $i$ and $j$ be parallel arrows in $Q$. Then the vertices $i$ and $j$ are connected by exactly two edges $a$ and $b$ in $\Inc(Q)$, with signs given by
 \[
 \sigma(a)=(-1)\sigma_Q(v,i)\sigma_Q(v,j) \quad \text{and} \quad \sigma(b)=(-1)\sigma_Q(v',i)\sigma_Q(v',j),
 \]
where $v$ and $v'$ are the vertices in $Q$ such that $\MiNu(i)=\{v,v'\}=\MiNu(j)$. Notice that $\sigma(a)=\sigma(b)$.

 \item[c)] If $i$ and $j$ are non-adjacent arrows in $Q$, then the vertices $i$ and $j$ are not adjacent in $\Inc(Q)$.
\end{itemize}
\end{definition}
See~\cite{jaJ2018} for an alternative and more general construction of $\Inc(Q)$ (compare also with the signed line graph construction in~\cite{tZ08,BS16}). The following lemma contains a graphical description of the quadratic form $q_Q$.

\begin{lemma}\label{L:08}
For any loop-less quiver $Q$ with $n$ arrows, the quadratic form $q_Q:\Z^n \to \Z$ of $Q$ is a unit form, and
\begin{itemize}
\itemsep=0.95pt
 \item[a)] $\widecheck{G}_{Q}=\Id-\Adj(\Inc(Q))$.
 \item[b)] $q_Q=q_{\Inc(Q)}$.
 \item[c)] $q_Q$ is non-negative.
 \item[d)] $q_Q$ is connected if and only if $Q$ is a connected quiver.
\end{itemize}
\end{lemma}
\begin{proof}
The non-negativity of $q_Q$ follows directly from the definition $G_{q_Q}=I(Q)^{\tr}I(Q)$. The diagonal entries of $I(Q)^{\tr}I(Q)$ are the squared norms of the columns $I_i$ of $I(Q)$. Since $Q$ has no loop, each column $I_i=\bas_{\sou(i)}-\bas_{\tar(i)}$ has squared norm two, which shows that $q_{Q}$ is a unit form.

\medskip
Observe that, by definition of $\Inc(Q)$, we have
\begin{equation*}
\Adj(\Inc(Q))= \left\{
\begin{array}{l l}
-I_i^{\tr}I_j, & \text{if $i < j$},\\
0, & \text{if $i \geq j$}.
\end{array} \right.
\end{equation*}
In particular, since $Q$ and $\Inc(Q)$ have no loop, we have
\[
2\Id-[\Adj(\Inc(Q))+\Adj(\Inc(Q))^{\tr}]=I(Q)^tI(Q),
\]
that is, $\widecheck{G}_Q=\Id-\Adj(\Inc(Q))$. Therefore $q_Q=q_{\Inc(Q)}$.

\medskip
By construction, the incidence bigraph $\Inc(Q)$ is connected if and only if $Q$ is connected. Then the last claim follows from Remark~\ref{R:02}.
\end{proof}

\subsection{The rank of an incidence matrix} \label{S(C):rk}

The following result is well known (cf.~\cite{tZ08}), and easy to prove.

\begin{lemma}\label{L:09}
For a connected loop-less quiver $Q$ we have $\Rnk(I(Q))=|Q_0|-1$.
\end{lemma}
\begin{proof}
Assume first that $Q$ is a tree (that is, that $|Q_1|=|Q_0|-1$), and proceed by induction on $m=|Q_0|$ (the cases $m \leq 2$ are clear). Assume that $m>2$ and take a vertex $v$ in $Q$ with degree one. Then there is a unique arrow $i$ in $Q$ containing $v$, and the restriction $Q^{(v)}$ is a tree. By induction hypothesis, the set of columns $\{I_j \mid j \in Q_1-\{i\}\}$ is linearly independent. Notice that $v$ does not belong to the support of $I_j$ for $j \neq i$, which implies that the whole set $\{I_j\}_{j \in Q_1}$ is linearly independent, which completes the induction step.

\medskip
Assume now that $Q$ is an arbitrary connected quiver with $n$ arrows, and choose a spanning tree $Q'$ of $Q$. By the first part of the proof, the set of columns $\{I_j\}$ with $j \in Q'_1$ is linearly independent. In particular $\Rnk(I(Q))\geq m-1$. Take now an arrow $i \in Q_1-Q'_1$, and let $v$ and $w$ be respectively the source and target of $i$. Then there is a (unique) walk $\alpha=i_1^{\epsilon_1}\cdots i_{\ell}^{\epsilon_{\ell}}$ in $Q'$ with starting vertex $v$ and ending vertex $w$. We have
\[
I_{\alpha}=\sum_{i=1}^{\ell}\epsilon_tI_{i_t}=\bas_{\sou(\alpha)}-\bas_{\tar(\alpha)}=\bas_v-\bas_w=I_i,
\]
which shows that $\Rnk(I(Q))\leq m-1$, hence the result.
\end{proof}

\begin{corollary}\label{C:10}
For a connected loop-less quiver $Q$ we have $\CRnk(q_{Q})=|Q_1|-|Q_0|+1$.
\end{corollary}
\begin{proof}
Since $q_Q$ is a non-negative unit form in $|Q_1|$ variables and $\Rnk(G_{q_Q})=\Rnk(I(Q))$ (see for instance~\cite{fZ99}), by the lemma above we have the result (for $\CRnk(q)=n-\Rnk(G_q)$).
\end{proof}

\subsection{Connection between the elementary transformations of $Q$ and $q_Q$} \label{S(C):trans}

\begin{lemma}\label{L:11}
Let $Q$ be a loop-less quiver with arrows $i \neq j$. Then, for $\epsilon \in \{+1,0,-1\}$, the linear transformation $T^{\epsilon}_{ij}$ is a flation for $q_Q$ if and only if $\mathcal{T}^{\epsilon}_{ij}$ is a quiver flation for $Q$.
\end{lemma}
\begin{proof}
Take $q=q_Q$. Recall that $T^{\epsilon}_{ij}$ is a flation for $q$ if $\epsilon=\sgn(q_{ij})$. On the other hand, $\mathcal{T}^{\epsilon}_{ij}$ is a flation for $Q$ if $\epsilon=-\epsilon_i\epsilon_j$, where $i^{\epsilon_i}j^{\epsilon_j}$ is a walk of $Q$, when the arrows $i$ and $j$ are adjacent, and $\epsilon=0$ otherwise. Using Lemma~\ref{L:08}, we simply compute $q_{ij}$ for each of the cases in Definition~\ref{D:07}:
\begin{itemize}
 \item[a)] Assume $\MiNu(i) \cap \MiNu(j)=\{v\}$. Then $q_{ij}=\sigma_Q(v,i)\sigma_Q(v,j)$, and $i^{-\sigma_Q(v,i)}j^{\sigma_Q(v,j)}$ is a walk of $Q$. Therefore $\sgn(q_{ij})=\sigma_Q(v,i)\sigma_Q(v,j)=-\epsilon_i\epsilon_j$.
 \item[b)] Assume $\MiNu(i) \cap \MiNu(j)=\{v,w\}$. Then $|q_{ij}|=2$ and $\sgn(q_{ij})=\sigma_Q(v,i)\sigma_Q(v,j)=\sigma_Q(w,i)\sigma_Q(w,j)$. As before, $i^{-\sigma_Q(v,i)}j^{\sigma_Q(v,j)}$ is a walk of $Q$, and $\sgn(q_{ij})=\sigma_Q(v,i)\sigma_Q(v,j)=-\epsilon_i\epsilon_j$.
 \item[c)] If $i$ and $j$ are non-adjacent, then $q_{ij}=0$ and $\epsilon=0$.
\end{itemize}
This completes the proof.
\end{proof}

\begin{proposition}\label{P:12}
Let $Q$ be a loop-less quiver, and consider an elementary transformation $\mathcal{A} \in \{\mathcal{V}_C,\mathcal{S}_{ij}$, $\mathcal{T}^{\epsilon}_{ij},\FS^{\epsilon}_{ij}\}$ for $Q$ (with $C$ a subset of arrows, and arrows $i \neq j$) with corresponding elementary linear transformation $A \in \{ V_C,S_{ij},T^{\epsilon}_{ij},\FSq^{\epsilon}_{ij}\}$. Then
\[
I(Q\mathcal{A})=I(Q)A \quad \text{and} \quad q_{Q\mathcal{A}}=q_QA.
\]
\end{proposition}
\begin{proof}
Notice that the claims for $\mathcal{A}=\mathcal{V}_C$ or $\mathcal{A}=\mathcal{S}_{ij}$ follow directly from the definitions, thus we only need to show the claims for $\mathcal{A}=\mathcal{T}^{\epsilon}_{ij}$.

\medskip
If $i$ and $j$ are non-adjacent arrows, there is nothing to prove. Assume that $i$ and $j$ are adjacent arrows, say $\MiNu(i)=\{v,w\}$ and $\MiNu(j)=\{v,w'\}$. Let $I'_1,I'_2,\ldots$ be the columns of the (vertex-arrow) incidence matrix $I(Q')$ where $Q'=Q\mathcal{T}^{\epsilon}_{ij}$. By definition (see~\ref{S(I):Quiv}), the new arrow $i'$ in $Q'$ satisfies $\MiNu(i')=\{w,w'\}$, and $I'_k=I_k$ for any arrow $k \in Q_1-\{i\}$. Moreover, observe that
\[
\sigma_Q(v,i)I'_i=\sigma_Q(v,i)I_i-\sigma_Q(v,j)I_j, \quad \text{that is,} \quad I'_i=I_i-\sigma_Q(v,i)\sigma_Q(v,j)I_j.
\]

On the other hand, $T^{\epsilon}_{ij}$ is a flation for $q_Q$ if $\epsilon=\sgn((q_Q)_{ij})$. Hence, using Lemma~\ref{L:08} we get $\epsilon=\sigma_Q(v,i)\sigma_Q(v,j)$. Take $I''=I(Q)T^{\epsilon}_{ij}$ with columns $I''_1,I''_2,\ldots$ Then, for $k \in Q_1$ we have
\begin{equation*}
I''_k =I(Q)T^{\epsilon}_{ij}(\bas_k) \left\{
\begin{array}{l l}
I(Q)\bas_k=I_k, & \text{if $k \neq i$},\\
I(Q)(\bas_i-\epsilon\bas_j)=I_i-\epsilon I_j, & \text{if $k=i$}.
\end{array} \right.
\end{equation*}
Therefore $I''_k=I'_k$ for all $k$, which completes the proof (cf. also~\cite[Proposition~5.2$(b)$ and Remark~5.2]{jaJ2018}).

\medskip
For the claims on quadratic forms, taking $q=q_Q$ and $q'=q_{Q\mathcal{A}}$, we have
\[
G_{q'}=I(Q\mathcal{A})^{\tr}I(Q\mathcal{A})=A^{\tr}I(Q)^{\tr}I(Q)A=A^{\tr}G_{q}A,
\]
that is, $q'=qA$.
\end{proof}

\begin{remark}\label{R:13}
For a subset $C$ of arrows in $Q$ we have
\[
q_{Q\mathcal{V}_C} \approx^{V_C} q_{Q},
\]
where $V_C$ is the point inversion of~\ref{S(I):trans} over the indices determined by $C$.
\end{remark}
\begin{proof}
Take $q=q_Q$, $q'=q_{Q\mathcal{V}_C}$ and $V=V_C$. Notice that $V^{\tr}TV$ is an upper (or lower) triangular matrix if and only if so is $T$. This observation and the equality $G_{q'}=V^{\tr}G_{q}V$ show that $\widecheck{G}_{q'}=V\widecheck{G}_{q}V$, that is, $q' \approx^{V} q$.
\end{proof}

\subsection{Admissibility} \label{S(C):adm}

The following observation justifies our considerations on $FS$-transformations.

\begin{lemma}\label{L:14}
Let $q(x)=\sum \limits_{1 \leq i \leq j \leq n}q_{ij}x_ix_j$ be a unit form, and take indices $i \neq j$ such that $|q_{ij}|\leq 1$. Consider the $FS$-transformation $\FSq=\FSq^{\epsilon}_{ij}$  for $q$, and take $q'=q\FSq$. Then $q' \approx^{\FSq} q$ if and only if $q_{ik}=0=q_{kj}$ for any integer $k$ such that $i<k<j$ or $j<k<i$.
\end{lemma}
\begin{proof}
Assume that $i<j$. Consider the matrix $\widecheck{G}_q$ partitioned in the following way, where $q_{i,\bullet}^{\tr}=(q_{i,i+1},\ldots\,$, $q_{i,j-1})$ and $q_{\bullet,j}^{\tr}=(q_{i+1,j},\ldots,q_{j-1,j})$,
\[
\widecheck{G}_q=\begin{pmatrix}
G_1 & y_1  & A & y_2 & B \\
0 & 1 & q_{i,\bullet}^{\tr} & q_{ij} & z^{\tr}_1 \\
0 & 0 & G_2 & q_{\bullet,j} & C \\
0 & 0 & 0 & 1 & z^{\tr}_2 \\
0 & 0 & 0 & 0 & G_3
\end{pmatrix},
\]
and where $G_1$, $G_2$ and $G_3$ are upper triangular (square) matrices with all diagonal entries equal to~$1$, with matrices $A, B, C$ and vectors $y_1,y_2,z_1,z_2$ of appropriate size. Observe that the composition $(T^{\epsilon}_{ij})^{\tr}\widecheck{G}_qT^{\epsilon}_{ij}$ has the following shape,
\[
(T^{\epsilon}_{ij})^{\tr}\widecheck{G}_qT^{\epsilon}_{ij}=
\begin{pmatrix}
G_1 & y_1-y_2  & A & y_2 & B \\
0 & 2-\epsilon q_{ij} & q_{i,\bullet}^{\tr} & q_{ij}-\epsilon & z^{\tr}_1-z^{\tr}_2 \\
0 & -q_{\bullet,j} & G_2 & q_{\bullet,j} & C \\
0 & -\epsilon & 0 & 1 & z^{\tr}_2 \\
0 & 0 & 0 & 0 & G_3
\end{pmatrix}.
\]
Swapping the $i$-th and $j$-th columns and rows of the matrix above, we get
\[
(\FSq^{\epsilon}_{ij})^{\tr}\widecheck{G}_q\FSq^{\epsilon}_{ij}=
\begin{pmatrix}
G_1 & y_2  & A & y_1-y_2 & B \\
0 & 1 & 0 & -\epsilon & z^{\tr}_2 \\
0 & q_{\bullet,j} & G_2 & -q_{\bullet,j} & C \\
0 & q_{ij}-\epsilon & q_{i,\bullet}^{\tr} & 2-\epsilon q_{ij} & z^{\tr}_1-z^{\tr}_2 \\
0 & 0 & 0 & 0 & G_3
\end{pmatrix}.
\]
Now, since $\epsilon=\textrm{sgn}(q_{ij})$ and $|q_{ij}|\leq 1$, then $q_{ij}-\epsilon=0$ (and $2-\epsilon q_{ij}=1$, as already argued in Remark~\ref{R:03}$(a)$). We conclude by observing that $\FSq^{\tr}\widecheck{G}_q\FSq$ is an upper triangular matrix if and only if $q_{i,\bullet}=0$ and $q_{\bullet,j}=0$. The case $j<i$ can be shown in a similar way.
\end{proof}

A $FS$-transformation $\FSq^{\epsilon}_{ij}$ for $q$ such that $q\FSq^{\epsilon}_{ij} \approx^{\FSq^{\epsilon}_{ij}} q$ will be called \textbf{$q$-admissible}.

\medskip
For an arrow $i \in Q_1$ in a loop-less quiver $Q$, consider the set $Q_1(i)$ of all arrows in $Q$ adjacent to arrow $i$, that is, $Q_1(i)=\{j \in Q_1 \mid \text{$j \neq i$ and $\MiNu(i) \cap \MiNu(j) \neq \emptyset$}\}$. We say that two arrows $i \neq j$ in $Q$ are \textbf{$Q$-admissible} (or that $\FS^{\epsilon}_{ij}$ is \textbf{$Q$-asmissible}) if we have
\[
k \notin Q_1(i) \cup Q_1(j), \qquad \text{for any arrow $k$ with $i<k<j$ or $j<k<i$}.
\]

\begin{corollary}\label{C:15}
Let $Q$ be a loop-less quiver with arrows $i \neq j$, and consider the flation $\FS=\FS^{\epsilon}_{ij}$ for $Q$, and the corresponding $FS$-transformation $\FSq=\FSq_{ij}^{\epsilon}$ for $q_Q$. Then $\FSq$ is $q_{Q}$-admissible if and only if $\FS$ is $Q$-admissible, and in that case
\[
q_{Q\FS} \approx^{\FSq} q_{Q}.
\]
\end{corollary}
\begin{proof}
Take $q_Q(x)=\sum \limits_{1 \leq i \leq j \leq n}q_{ij}x_ix_j$. By Lemma~\ref{L:14}, the transformation $\FSq$ is $q_Q$ admissible if and only if $q_{ik}=0=q_{kj}$ for all $k$ such that $i<k<j$ or $j<k<i$. Correspondingly, using Lemma~\ref{L:08} and the definition of incidence graph $\Inc(Q)$, this means that
\[
k \notin Q_1(i) \cup Q_1(j), \qquad \text{for any arrow $k$ with $i<k<j$ or $j<k<i$},
\]
that is, $\FSq$ is $q_Q$ admissible if and only if $\FS$ is $Q$-admissible.
\end{proof}

\subsection{Iterated transformations} \label{S(C):ite}

Recall that by iterated $FS$-(linear) transformation $\FSq$ for $q$ we mean a composition of the form $\FSq=\FSq^{\epsilon_1}_{i_1j_1}\cdots \FSq^{\epsilon_r}_{i_rj_r}$ such that $\FSq^{\epsilon_t}_{i_tj_t}$ is a $FS$-transformation for $q^{t-1}$, where $q^0=q$ and we take recursively $q^{t}=q^{t-1}\FSq_{i_tj_t}$, and such that each $q^t$ is a unit form for $t=1,\ldots,r$. If each $\FSq_{i_tj_t}^{\epsilon_t}$ is $q^{t-1}$-admissible, we say that $\FSq$ is a \textbf{$q$-admissible iterated $FS$-(linear) transformation}. In this case we have $q^r \approx^{\FSq} q$.

\medskip
An \textbf{iterated $FS$-transformation} $\FS$ of a loop-less quiver $Q$ is a concatenation $\FS=\FS_{i_1j_1}^{\epsilon_1}\cdots \FS^{\epsilon_r}_{i_rj_r}$ of $FS$-transformations such that if we take inductively $Q^{0}=Q$ and $Q^{t}=Q^{t-1}\FS^{\epsilon_t}_{i_tj_t}$, then the pair of arrows $i_t$ and $j_t$ are not parallel in $Q^{t-1}$ for $t=1,\ldots,r$. The expression $Q\FS$ denotes the final (loop-less) quiver $Q^r$. If in each step $t$, the $FS$-transformation $\FS^{\epsilon_t}_{i_tj_t}$ is $Q^{t-1}$-admissible, then the iterated $FS$-transformation $\FS$ is called \textbf{$Q$-admissible}. By definition and Corollary~\ref{C:15}, $\FS$ is $Q$-admissible if and only if $\FSq$ is $q_Q$ admissible, and in that case
\[
q_{Q\FS} \approx^{\FSq}q_Q.
\]
Notice that if $\FS$ is a $Q$-admissible iterated $FS$-transformation, and $\FS'$ is a $Q\FS$-admissible iterated $FS$-transformation, then the concatenation $\FS\FS'$ is a $Q$-admissible iterated $FS$-transformation.

\medskip
Recall that a \textbf{maximal (quiver) star} $\Star_n$ with $n$ arrows is a tree quiver with $n+1$ vertices (and arbitrary direction of its arrows), such that there exists a vertex $v_0$ incident to all arrows of the quiver (called \textbf{center of the star}). One of our main combinatorial problems is to show that any tree quiver can be brought to a maximal star via an admissible iterated $FS$-transformation. We need the following key preliminary observation.

\begin{lemma}\label{L:16}
For any maximal (quiver) star $\Star$ and any vertex $v \in \Star_0$, there exists a $\Star$-admissible iterated $FS$-transformation $\FS$ such that $\Star\FS$ is a maximal star having $v$ as its center.
\end{lemma}

\begin{proof}
Let $1,\ldots,{n}$ be the arrows of the maximal star $\Star$, all of them having in common the center of the star $v_0$, and assume $n \geq 2$. The rest of vertices $v_1,\ldots,v_n$ of $\Star$ are enumerated such that $v_t$ is incident to the arrow $t$ for $t=1,\ldots,n$. We show that the following is a $\Star$-admissible iterated $FS$-transformation,
\[
\mathcal{W}=\FS_{2,1}\FS_{3,2}\cdots \FS_{n,n-1}.
\]
For $t=1,\ldots,n$, take $Q^{t}$ to be the tree quiver with same set of vertices than $\Star$, given by
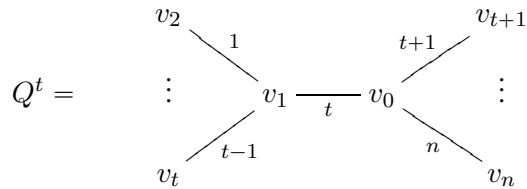
\[
\xymatrix@R=1pc{
& v_2 \ar@{-}[rd]^-{1} & & & v_{t+1} \ar@{-}[ld]_-{t+1} \\
Q^{t}= & \vdots & v_1 \ar@{-}[r]_-{t} & v_0 & \vdots \\
& v_{t} \ar@{-}[ru]_-{t-1} & & & v_n \ar@{-}[lu]^-{n} }
\]
(the direction of arrows, not shown in the diagram, is irrelevant for the proof). Then $Q^{1}=\Star$, and clearly $\FS_{t+1,t}$ is a $Q^{t}$-admissible $FS$-transformation (admissibility holds since $t+1$ and $t$ are consecutive arrows). Observe that $Q^{t+1}=Q^{t}\FS_{t+1,t}$, and thus by definition $Q^{n}=\Star \mathcal{W}$. Notice also that $Q^{n}$ is a maximal star with center the vertex $v_1$, that $\mathcal{W}$ is a $Q^{n}$-admissible iterated $FS$-transformation, and that the first arrow $1$ in $Q^{n}$ joins the vertices $v_1$ and $v_2$. In particular, the quiver $(\Star\mathcal{W})\mathcal{W}$ is a maximal star with center $v_2$, whose first arrow joins vertices $v_2$ and $v_3$.

Now, if $v=v_0$ is the original center of the star, there is nothing to do. If $v=v_t$ for some $1 \leq t \leq n$, then we may repeat the above construction $t$ times to get a maximal star $\Star \mathcal{W}^{t}$ with center the vertex $v_t$, as wanted, where $\mathcal{W}^t$ denotes the concatenation of $t$ copies of $\mathcal{W}$.
\end{proof}

\begin{proposition}\label{P:17}
For any tree quiver $Q$ with selected vertex $v$, there exists a $Q$-admissible iterated $FS$-transformation $\FS$ such that $Q\FS$ is a maximal star with center the vertex $v$.
\end{proposition}
\begin{proof}
We proceed by induction on the number of arrows $n=|Q_1|$ of the tree $Q$. For $n=1$ there is nothing to show. For $n=2$ the tree $Q$ is a star, and we may change the position of its center as in the Lemma above. Hence, we may assume that $n \geq 3$ and that the claim holds for all trees with less than $n$ arrows.

Let $n$ be the maximal arrow in $Q$ (relative to the total order $\leq$ in $Q_1$) and take $\MiNu(n)=\{v,w\}$. Let $Q'$ be the quiver obtained from $Q$ by deleting the arrow $n$. Then $Q'$ is the disjoint union of exactly two tree quivers, one containing vertex $v$ and denoted by $Q^v$, and one containing vertex $w$ and denoted by $Q^w$ . The sets $Q^v_1$ and $Q^w_1$ inherit the total order from $Q_1$.

Observe that, by the maximality of $n$, any $Q^v$-admissible iterated $FS$-transformation is also $Q$-admissible, and similarly for $Q^w$. We distinguish two cases:

\medskip
\noindent \textbf{Case 1.} Assume first that $|Q^w_0|=1$ (that is, that $w$ is a leaf in $Q$). By induction hypothesis, we may assume that $Q^v$ is a maximal star with center $v$. Then $Q$ is a maximal star. We proceed analogously if $|Q^v_0|=1$.

\medskip
\noindent \textbf{Case 2.} Assume that $|Q^w_0|>1$ and $|Q^v_0|>1$, and that the second largest arrow $n-1$ in $Q$ belongs to $Q^v$. By induction hypothesis, we may assume that $Q^v$ is a maximal star with center $v$. In particular, $n-1$ and $n$ are adjacent arrows and $n-1$ is a pendant arrow in $Q$. Then $\FS_{n-1,n}$ is a $Q$-admissible $FS$-transformation, and the maximal arrow $n$ in $Q'=Q\FS_{n-1,n}$ is a pendant arrow in $Q'$. Apply then Case~1 to the tree quiver $Q'$.

\medskip
To complete the proof, use the Lemma above to change the center of the resulting maximal star, as desired.
\end{proof}

\begin{remark}\label{R:18}
For a loop-less quiver $Q$, all $Q$-admissible iterated $FS$-transformations are reversible. To be precise, if $\FS=\FS^{\epsilon_1}_{i_1j_1}\cdots\FS^{\epsilon_r}_{i_rj_r}$ is a $Q$-admissible iterated $FS$-transformation and $Q'=Q\FS$, then
\[
\FS^{-1}:=\FS^{-\epsilon_r}_{j_ri_r}\cdots \FS^{-\epsilon_1}_{j_1i_1},
\]
is a $Q'$-admissible $FS$-transformation, and $Q=Q'\FS^{-1}$.
\end{remark}
\begin{proof}
The claim follows inductively from Remark~\ref{R:06}$(b)$.
\end{proof}

\subsection{Strong congruence among positive unit forms of Dynkin type $\A_n$} \label{S(C):proof}

The following proposition is one of the goals in~\cite{jaJ2018} (see~\cite[Theorem~5.5$(c)$ and Corollary~6.6]{jaJ2018}). Here we give a short proof. For $c \geq 0$, define the \textbf{canonical $c$-extension linear quiver} $\overrightarrow{\A}_n^{(c)}$ as a quiver obtained from the linear quiver $\overrightarrow{\A}_n=\overrightarrow{\A}_n^{(0)}=\xymatrix{v_1 \ar[r]^-{1} & v_2 \ldots v_n \ar[r]^-{n} & v_{n+1}}$ by adding $c$ arrows from $v_{n+1}$ to $v_1$:
\[
\overrightarrow{\A}_n^{(c)}=\xymatrix{v_1 \ar[r]^-{1} & v_2 \ar[r]^-{2} & v_3 \ldots v_{n-1} \ar[r]^-{n-1} & v_n \ar[r]^-{n} & v_{n+1} \ar@<2.5ex>@/^20pt/[llll]^-{n+c} \ar@<.5ex>@/^20pt/[llll]_-{n+1}^-{\cdots} \\ {} }
\]

Observe that the incidence graph of $\overrightarrow{\A}_n^{(c)}$ is the canonical $c$-vertex extension $\widehat{\A}_n^{(c)}$ of Dynkin type $\A_{n}$, as defined by Simson in~\cite[Definition~2.2]{dS16a},
\[
\Inc(\overrightarrow{\A}_n^{(c)})=\widehat{\A}_n^{(c)}.
\]

\begin{proposition}\label{P:19}
Let $q$ be a connected non-negative unit form of Dynkin type $\A_n$. Then there is a connected loop-less quiver $Q$ such that $q=q_Q$.
\end{proposition}
\begin{proof}
By Theorem~\ref{T:04}, there is an iterated flation $T=T^{\epsilon_1}_{i_1j_1}\cdots T^{\epsilon_r}_{i_rj_r}$ for $q$ such that $qT=q_{\widehat{\A}_n^{(c)}}$ is the canonical $c$-extension of  $q_{\A_n}$ (for $c \geq 0$ the corank of $q$, see~\cite{dS16a,SZ17}).

\medskip
Take $Q$ as the unique quiver satisfying
\[
I(Q)=I(\overrightarrow{\A}_n^{(c)})T^{-1},
\]
and notice that for any $x \in \Z^n$ we have
\begin{eqnarray*}
 q_Q(x)\hspace{-2mm}&=&\hspace{-2mm}\frac{1}{2} x^{\tr}I(Q)^{\tr}I(Q)x =\frac{1}{2}
                                  x^{\tr}T^{-\tr}I(\overrightarrow{\A}_n^{(c)})^{\tr}I(\overrightarrow{\A}_n^{(c)})T^{-1}x=\\
  &=\hspace{-2mm}& q_{\widehat{\A}_n^{(c)}}(T^{-1}x)=q(x),
\end{eqnarray*}
that is, $q=q_Q$.
\end{proof}

Now we are ready to prove one of the main Gram classification results of the paper.

\begin{theorem}\label{T:20}
Let $q$ be a connected positive unit form of Dynkin type $\A_n$. If $q'$ is a unit form with $q' \sim q$, then there exists a composition of $q$-admissible iterated $FS$-linear transformations and sign changes $B$, such that
\[
q' \approx^B q.
\]
\end{theorem}
\begin{proof}
Let $Q$ be a tree quiver such that $q=q_Q$ as in Proposition~\ref{P:19}. By Proposition~\ref{P:17}, there is a $Q$-admissible iterated $FS$-transformation $\FS$ such that $Q\FS$ is a maximal star. Take an arrow inversion $\mathcal{V}$ such that $\Star=Q\FS \mathcal{V}$ has all arrows pointing away from the star center. Denote by $\FSq$ the $FS$-linear transformation for $q$ corresponding to $\FS$, and by $V$ the point inversion corresponding to $\mathcal{V}$, so that $qTV=q_{\Star}$, by Proposition~\ref{P:12}.

\medskip
Now, if $q' \sim q$ then $q'$ is a connected positive unit form (by Lemma~\ref{L:01}). As before, there is a composition of $q'$-admissible iterated $FS$-linear transformations $T'$, and a point inversion $V'$, such that $q'T'V'=q_{\Star}=qTV$.  Considering the linear transformation $B=TV(V')^{-1}(T')^{-1}$, by Corollary~\ref{C:15} and Remark~\ref{R:13} we have $q' \approx^B q$, which completes the proof.
\end{proof}

\section{Combinatorial Coxeter analysis of unit forms of Dynkin type $\A_n$} \label{S(A)}

We begin this section giving some combinatorial definitions, in particular the notion of \emph{inverse of a quiver}, which will be used for the Coxeter analysis of non-negative unit forms of Dynkin type $\A_n$. Let $Q=(Q_0,Q_1,\sou,\tar)$ be a loop-less quiver, with a fixed total ordering $\leq$ of its arrows, and consider the sets
 \[
 Q_1^<(v,i)=\{j \in Q_1 \mid \text{$j<i$ and $v \in \MiNu(i) \cap \MiNu(j)$}\},
 \]
 for each vertex $v$ and each arrow $i$ of $Q$.

\begin{definition}\label{D:22}
Let $\alpha=i_0^{\epsilon_0}i_1^{\epsilon_1}\cdots i_{\ell}^{\epsilon_{\ell}}$ be a non-trivial walk of $Q$.
\begin{itemize}
 \item[a)] We say that the walk $\alpha$  is \textbf{minimally decreasing} if
\[
i_{t+1}=\max Q_1^{<}(v_t,i_t), \quad  \text{for $t=0,\ldots,\ell-1$}.\vspace*{-1.8mm}
\]
 \item[b)] If $\alpha$ is minimally decreasing, we say that $\alpha$ is \textbf{left complete} if whenever $\beta \alpha$ is minimally decreasing for some walk $\beta$, then $\beta$ is a trivial walk. Similarly,  $\alpha$ is \textbf{right complete} if whenever $\alpha \beta$ is minimally decreasing for some walk $\beta$, then $\beta$ is a trivial walk. A left and right complete minimally decreasing walk will be called \textbf{structural (decreasing) walk}.
\end{itemize}
\end{definition}

We will mainly consider the following particular minimally decreasing walks. For an arrow $i$ with there are exactly two right complete minimally descending walks starting with arrow $i$, one starting at vertex $\sou(t)$ and denoted by $\alpha^-_Q(i^{+1})$, and one starting at vertex $\tar(i)$ and denoted by $\alpha^-_Q(i^{-1})$. To be precise, if
\[
 \alpha_Q^-(i^{+1})=(v_{-1},i_0,v_{0},i_1,v_1,\ldots,v_{\ell-1},i_{\ell},v_{\ell}),
\]
then $v_{-1}=\sou(i)$, $i_0=i$, $i_{t+1}=\max Q_1^{<}(v_t,i_t)$ for $t=0,\ldots,\ell-1$,  and $Q_1^<(v_{\ell},i_{\ell})=\emptyset$.

\medskip
 Similarly, if
\[
 \alpha_Q^-(i^{-1})=(v_{-1},i_0,v_{0},i_1,v_1,\ldots,v_{\ell-1},i_{\ell},v_{\ell}),
\]
then $v_{-1}=\tar(i)$, $i_0=i$, $i_{t+1}=\max Q_1^{<}(v_t,i_t)$ for $t=0,\ldots,\ell-1$,  and $Q_1^<(v_{\ell},i_{\ell})=\emptyset$. For a vertex $v$ we denote by $\alpha_Q^-(v)$ the unique structural decreasing walk starting at $v$.

\medskip
Consider dually the minimally increasing walks $\alpha_Q^+(i^{\pm 1})$ and $\alpha_Q^+(v)$.

\begin{definition}\label{D:inv} Define a new quiver $Q^*=(Q^*_0,Q^*_1,\sou^*,\tar^*)$ having the same set of vertices $Q_0^*=Q_0$ than $Q$, and the same number of arrows $|Q_1^*|=|Q_1|$, and such that each arrow $i$ in $Q$ corresponds to an arrow $i^*$ in $Q^*$, given by
 \[
 \sou^*(i^*)=\tar(\alpha_Q^-(i^{-1})), \qquad \text{and} \qquad \tar^*(i^*)=\tar(\alpha_Q^-(i^{+1})).
 \]
The order of the set of arrows of $Q_1^*$ corresponds to the order in $Q_1$ (that is, $i^* \leq j^*$ in $Q^*_1$ if and only if $i \leq j$ in $Q_1$).
\end{definition}

The quiver $Q^*$ defined above will be referred to as \textbf{inverse quiver} of $Q$, and we will use the notation $Q^{-1}=Q^*$ (we will drop the asterisk $*$ on arrows of $Q^{-1}$ when the context allows it). Proposition~\ref{P:24} below, for which we need the following technical result, justifies our definitions.

\subsection{A technical lemma} \label{S(A):Inv}

Recall that the columns of the (vertex-arrow) incidence matrix $I(Q)$ of $Q$ are denoted by $I_i=\sou(i)-\tar(i) \in \Z^{|Q_0|}$ for an arrow $i$ of $Q$. The columns of $I(Q^{-1})$ will be denoted by $I_i^{-1}$ for an arrow $i$ of the inverse quiver $Q^{-1}$. For convenience, we consider the function $\langle-,- \rangle:Q_1 \times Q_1 \to \Z$ given by $\langle i,j \rangle=I^{\tr}_iI_j$ for arrows $i$ and $j$ in $Q$.

\begin{lemma}\label{L:23}
Let $Q$ be a loop-less quiver. Define an auxiliary function $\Xi:Q_0 \times Q_1 \to \Z^{|Q_0|}$, given for a vertex $v$ and an arrow $k$ of $Q$ by,
\[
\Xi(v,k)=\sum_{i \in Q_1^<(v,k)}I^{-1}_i\frac{\langle i,k \rangle}{|\langle i,k \rangle|},
\]
where as usual the sum is zero when the set of indices is empty. Then the following assertions hold:
\begin{itemize}
\itemsep=0.98pt
 \item[a)] If $j=\max Q_1^<(v,k)$ and $\MiNu(j)=\{v,w\}$, then $\Xi(v,k)=\frac{\langle j,k \rangle}{|\langle j,k \rangle|}\left[I_j-\Xi(w,j) \right]$.
 \item[b)] For any arrow $k$ with $\MiNu(k)=\{v,w\}$ we have $\Xi(v,k)+\Xi(w,k)=I_k-I_k^{-1}$.
\end{itemize}
In particular, for any arrow $k$ in $Q$ the following recursive formula for $I_k^{-1}$ holds,
\[
I_k^{-1}=I_k-\sum_{i<k}I_i^{-1}\langle i,k \rangle.
\]
\end{lemma}
\begin{proof}
We proceed by induction on the totally ordered arrows $Q_1$. Observe that if $k$ is minimal in $Q_1$, then $\Xi(\sou(k),k)=0=\Xi(\tar(k),k)$ and $I^{-1}_k=I_k$, therefore all claims hold in this case. For simplicity, for adjacent arrows $j$ and $k$ we take
\[
\sigma(j,k)=\frac{\langle j,k \rangle}{|\langle j,k \rangle|} \in \{\pm 1\}.
\]

To verify $(a)$ for an arrow $k$, assume that $(b)$ is satisfied for all arrows smaller than $k$. Then, since $j=\max Q_1^<(v,k)$, we have $Q_1^<(v,k)=\{j\} \cup Q_1^<(v,j)$, and therefore
\begin{eqnarray}
\Xi(v,k) & = & I^{-1}_{j}\sigma(j,k)+\sum_{i \in Q_1^<(v,j)}I^{-1}_i\sigma(i,k)   \nonumber \\
& = & I^{-1}_{j}\sigma(j,k)+\sum_{i \in Q_1^<(v,j)}I^{-1}_i\sigma(i,j)\sigma(j,k) \label{EqOne} \\
& = & \sigma(j,k)\left[ I_{j}^{-1}+\Xi(v,j) \right] \nonumber \\
& = & \sigma(j,k)\left[ I_{j}^{-1}+I_{j}-I_{j}^{-1}-\Xi(w,j) \right]  \label{EqTwo} \\
& = & \sigma(j,k)\left[ I_{j}-\Xi(w,j) \right], \nonumber
\end{eqnarray}
where the equality~(\ref{EqOne}) holds since $\sigma(i,j)\sigma(j,k)=\sigma(i,k)$ whenever $\MiNu(i)\cap \MiNu(j) \cap \MiNu(k) \neq \emptyset$, and the equality~(\ref{EqTwo}) holds applying $(b)$ to the arrow $j$.

\medskip
To show $(b)$, assume the claim holds for all arrows smaller than $k$, and that $(a)$ holds for all arrows smaller than or equal to $k$. Consider the minimally descending walk $\alpha=\alpha^-_Q(k^{-1})$ of $i_0=k$ as in the definition above,
\[
\alpha=i_0^{\epsilon_0}\cdots i_r^{\epsilon_r},
\]
(in particular, $\epsilon_0=-1$ since $\tar(k)=v_{-1}$). For simplicity, for such a walk and for $t=1,\ldots,r$ define
\[
\sigma(i_t,i_{t-1},\ldots,i_1,i_0)=(-1)^{t+1}\frac{\langle i_{t},i_{t-1} \rangle \; \langle i_{t-1},i_{t-2} \rangle \cdots \langle i_{1},i_{0} \rangle}{|\langle i_{t},i_{t-1} \rangle \; \langle i_{t-1},i_{t-2} \rangle \cdots \langle i_{1},i_{0} \rangle|} \in \{\pm 1\}.
\]
Applying recursively $(a)$ we get $\Xi(v_0,i_0)=\sum_{t=1}^rI_{i_t}\sigma(i_t,\ldots,i_0)$. However, since $\sou(i_0)=v_0$, we have $\sigma(i_t,\ldots,i_0)=-\sigma(v_t,i_t)=\epsilon_t$, and therefore the above expression for $\Xi(v_0,i_0)$ is a telescopic sum, that is,
\[
\Xi(v_0,i_0)=I_{i_1}\epsilon_1+\ldots+I_{i_r}\epsilon_r=\bas_{v_0}-\bas_{\xi(v_0,i_0)}.
\]
Proceeding similarly for $\Xi(w_0,i_0)$, where $w_0=\tar(i_0)$, we find that
\[
\Xi(w_0,i_0)=\bas_{\xi(w_0,i_0)}-\bas_{w_0}.
\]
Since our definition of $I^{-1}_{i_0}$ is $I^{-1}_{i_0}=\bas_{\xi(v_0,i_0)}-\bas_{\xi(w_0,i_0)}$, these equations yield the result,
\[
\Xi(v_0,i_0)+\Xi(w_0,i_0)=\bas_{v_0}-\bas_{\xi(v_0,i_0)}+\bas_{\xi(w_0,i_0)}-\bas_{w_0}=I_{i_0}-I_{i_0}^{-1}.
\]

Now, to verify the last claim, for an arrow $k$ with $\MiNu(k)=\{v,w\}$ consider the sets
\[
X=Q_1^<(v,k)-Q_1^<(w,k), \quad Y=Q_1^<(w,k)-Q_1^<(v,k) \quad \text{and} \quad Z=Q_1^<(v,k)\cap Q_1^<(w,k).
\]
Then, applying $(b)$, we get
\begin{eqnarray}
I_k^{-1} & = & I_k-\sum_{i \in Q_1^<(v,k)}I^{-1}_i\sigma(i,k)-\sum_{i \in Q_1^<(w,k)}I^{-1}_i\sigma(i,k)  \nonumber \\
& = &I_{k}-\left(\sum_{i \in X\cup Y}I^{-1}_i\frac{\langle j,k \rangle}{|\langle j,k \rangle|}+2\sum_{i \in Z}I^{-1}_i\frac{\langle j,k \rangle}{|\langle j,k \rangle|} \right) \nonumber \\
& = & I_k-\sum_{i<k}I_i^{-1}\langle i,k \rangle,  \nonumber
\end{eqnarray}
where the last equality holds since $|\langle i,k \rangle|=1$ if $i \in X \cup Y$, and $|\langle i,k \rangle|=2$ if $i \in Z$ (observe that $Z$ is the set of arrows smaller than $k$ that are parallel to $k$), and $\langle i,k \rangle=0$ if $i \notin X \cup Y \cup Z$. This completes the proof.
\end{proof}

\subsection{Gram matrices of a quiver and its inverse} \label{S(A):InvDos}

We recall that the inverse quiver of a loop-less quiver $Q$ is given in Definition~\ref{D:inv}, and is denoted by $Q^{-1}$.

\begin{proposition}\label{P:24}
 Let $Q$ be a loop-less quiver with inverse quiver $Q^{-1}$ and upper triangular Gram matrices $\widecheck{G}_Q$ and $\widecheck{G}_{Q^{-1}}$. Then
 \[
 I(Q^{-1})=I(Q)\widecheck{G}^{-1}_Q \qquad \text{and} \qquad \widecheck{G}_Q\widecheck{G}_{Q^{-1}}=\Id.
 \]
\end{proposition}

\begin{proof}
Take $G=\widecheck{G}_Q$ and let $A$ be the (upper) triangular adjacency matrix $\Adj(\Inc(Q))$ of the incidence bigraph $\Inc(Q)$ of $Q$ (hence $G+G^{\tr}=I(Q)^{\tr}I(Q)=2\Id-(A+A^{\tr})$, cf.~\ref{S(C):Iqf}).  We proceed by induction on the number of arrows $m=|Q_1|$ of $Q$ (the claim is trivial for $m=1$). Let $1,\ldots,m$ be the ordered arrows of $Q$, and denote by $\widetilde{Q}$ the quiver obtained from $Q$ by removing the maximal arrow~$m$. Notice that, by maximality, the inverse $\widetilde{Q}^{-1}$ is obtained from $Q^{-1}$ by removing its maximal arrow. Then
\[
I(\widetilde{Q})=[I_{1}|\cdots |I_{{m-1}}], \quad \text{and} \quad G=\begin{pmatrix} \widetilde{G}&-V\\0&1 \end{pmatrix},
\]
where $\widetilde{G}=\widecheck{G}_{\widetilde{Q}}$ and $V$ is the column vector of size $m-1$ given by $V=(\langle {t},m \rangle)_{t=1}^{m-1}$ (since $G=\Id-A$, cf. Lemma~\ref{L:08}).  Observe that
\[
G^{-1}=\begin{pmatrix} \widetilde{G}^{-1}&\widetilde{G}^{-1}V\\0&1 \end{pmatrix},
\]
therefore, by induction hypothesis,
\begin{eqnarray}
I(Q)G^{-1} & = & [I(\widetilde{Q})|I_{m}]\begin{pmatrix} \widetilde{G}^{-1}&\widetilde{G}^{-1}V\\0&1 \end{pmatrix} \nonumber \\
& = & [I(\widetilde{Q})\widetilde{G}^{-1}|I(\widetilde{Q})\widetilde{G}^{-1}V+I_
{m}] \nonumber \\[4pt]
& = & [I(\widetilde{Q}^{-1}) | I_{m}+I(\widetilde{Q}^{-1})V].  \nonumber
\end{eqnarray}
We have shown in Lemma~\ref{L:23} that
\[
I_{m}+I(\widetilde{Q}^{-1})V=I_{m}-\sum_{t=1}^{m-1}I^{-1}_{t}\langle {t},m \rangle=I_{m}^{-1},
\]
that is, $I(Q)G^{-1}=[I(\widetilde{Q}^{-1})|I^{-1}_{m}]=I(Q^{-1})$, as wanted.

\medskip
Now we show that $\widecheck{G}_Q\widecheck{G}_{Q^{-1}}=\Id$. By the first claim, observe that
\begin{eqnarray}
\widecheck{G}_{Q^{-1}}+\widecheck{G}_{Q^{-1}}^{\tr}&=&G_{Q^{-1}}=I(Q^{-1})^{\tr}I(Q^{-1}) \nonumber \\
&=&\widecheck{G}_{Q}^{-\tr}I(Q)^{\tr}I(Q)\widecheck{G}_{Q}^{-1}=\widecheck{G}_{Q}^{-\tr}G_Q\widecheck{G}_{Q}^{-1} \nonumber \\[3pt]
& = & \widecheck{G}_{Q}^{-\tr}(\widecheck{G}_{Q}+\widecheck{G}_{Q}^{\tr})\widecheck{G}_{Q}^{-1} = \widecheck{G}_{Q}^{-\tr}+\widecheck{G}_{Q}^{-1}. \nonumber
\end{eqnarray}
This completes the proof since both $\widecheck{G}_{Q^{-1}}$ and $\widecheck{G}_{Q}^{-1}$ are upper triangular matrices.
\end{proof}

As consequence we have the following important observation.

\begin{corollary}\label{C:25}
For a loop-less quiver $Q$ the following assertions hold.
\begin{itemize}
 \item[a)] The inverse quiver $Q^{-1}$ has no loop, and
\[
Q=(Q^{-1})^{-1}.
\]
 \item[b)] The quiver $Q$ is connected if and only if $Q^{-1}$ is connected.
 \item[c)] Inverses commute with arrow inversions, that is, if $C$ is a set of arrows in $Q$, and $C^*$ is the set of corresponding arrows in $Q^{-1}$, then $(Q\mathcal{V}_C)^{-1}=Q^{-1}\mathcal{V}_{C^*}$.
\end{itemize}
\end{corollary}

\begin{proof}
For $(a)$, if $Q^{-1}$ has a loop, then its (vertex-arrow) incidence matrix $I(Q^{-1})$ has a zero column, say $I_i=0$. Then the $i$-th diagonal entry of $G_{Q^{-1}}=I(Q^{-1})^{\tr}I(Q^{-1})$ is zero. But $G_{Q^{-1}}=\widecheck{G}_Q^{-1}+\widecheck{G}_Q^{-\tr}$, which means that the $i$-th diagonal entry of $\widecheck{G}_Q^{-1}$ is zero. This is impossible since $\widecheck{G}_Q^{-1}$ is invertible. It follows that the inverse quiver $Q^{-1}$ has no loop. Now, applying Proposition~\ref{P:24}, we obtain
\[
I((Q^{-1})^{-1})=I(Q^{-1})\widecheck{G}_{Q^{-1}}^{-1}=[I(Q)\widecheck{G}_Q^{-1}]\widecheck{G}_Q=I(Q),
\]
that is, $Q=(Q^{-1})^{-1}$. This completes the proof of $(a)$.

\medskip
For $(b)$, notice that the construction of the inverse $Q^{-1}$ involves only vertices and arrows inside the connected components of $Q$, that is, if $Q=Q^1 \cup Q^2$, then
\[
Q^{-1}=(Q^1)^{-1} \cup (Q^2)^{-1}.
\]
Using $(a)$, this shows that $Q$ is connected if and only if so is $Q^{-1}$.

\medskip
For $(c)$, by applying Propositions~\ref{P:12} and~\ref{P:24} we obtain
\begin{eqnarray}
I((Q\mathcal{V}_C)^{-1})&=&I(Q\mathcal{V}_C)\widecheck{G}^{-1}_{Q\mathcal{V}_C}=I(Q)V_C(V_C^{\tr}\widecheck{G}_QV_C)^{-1} \nonumber \\
&=&I(Q)\widecheck{G}_Q^{-1}V_C^{-\tr}=I(Q^{-1})V_C=I(Q^{-1}\mathcal{V}_{C^*}). \nonumber
\end{eqnarray}
Here we use the equalities $\widecheck{G}_{Q\mathcal{V}_C}\!=\!V_C^{\tr}\widecheck{G}_QV_C\,$(cf.$\,$Remark~\ref{R:13}) and $V_C^{-\tr}\!=\!V_C.\,$Clearly$\,I((Q\mathcal{V}_C)^{-1})\!=I(Q^{-1}\mathcal{V}_{C^*})$ implies $(Q\mathcal{V}_C)^{-1}=Q^{-1}\mathcal{V}_{C^*}$. This finishes the proof of $(c)$.
\end{proof}

\subsection{A combinatorial formula for the Coxeter matrix} \label{S(A):Cox}

The \textbf{Coxeter matrix} associated to a unit form $q$ on $n$ variables is the $n \times n$ matrix given by
\[
\Cox_q=-\widecheck{G}_q^{\tr}\widecheck{G}_q^{-1}.
\]
The characteristic polynomial of $\Cox_q$, given by $\cox_q(\va)=\det(\Id\va-\Cox_q)$, is called \textbf{Coxeter polynomial} of $q$. The \textbf{Coxeter number} $\coxN_q$ of $q$ is the minimal natural number $m$ such that $\Cox_q^m=\Id$, if such number exists, and $\coxN_q=\infty$ otherwise (cf.~\cite{dS20} for these and related definitions). The following result may be found in~\cite[Proposition~4.2]{dS20} in a wider context.

\begin{lemma}\label{L:26}
Let $q$ and $q'$ be unit forms. If $q' \approx^Bq$, then
\[
\Cox_{q'}=B^{\tr}\Cox_qB^{-\tr}.
\]
In particular, $\cox_{q'}(\va)=\cox_q(\va)$ and $\coxN_{q'}=\coxN_q$.
\end{lemma}

\begin{proof}
By hypothesis we have $\widecheck{G}_{q'}=B^{\tr}\widecheck{G}_qB$, therefore $\widecheck{G}_{q'}^{-1}=B^{-1}\widecheck{G}^{-1}_qB^{-\tr}$, and
\[
\Cox_{q'}=-\widecheck{G}_{q'}^{\tr}\widecheck{G}_{q'}^{-1}=-(B^{\tr}\widecheck{G}_q^{\tr}B)(B^{-1}\widecheck{G}_q^{-1}B^{-\tr})=B^{\tr}\Cox_qB^{-\tr}.
\]
Finally, both the characteristic polynomial and the order of a square matrix are similarity invariants, therefore the remaining two equalities hold.
\end{proof}

We now give a combinatorial expression for the Coxeter matrix of some unit forms.

\begin{theorem}\label{T:27}
For a loop-less quiver $Q$, the following formula for the Coxeter matrix $\Cox_{q_Q}$ of the unit form $q_Q$ holds,
\[
\Cox_{q_Q}=\Id-I(Q)^{\tr}I(Q^{-1}).
\]
\end{theorem}
\begin{proof}
By Proposition~\ref{P:24} we have
\begin{eqnarray}
\Id-I(Q)^{\tr}I(Q^{-1})&=&\Id-I(Q)^{\tr}I(Q)\widecheck{G}_Q^{-1} = \Id-G_Q\widecheck{G}_Q^{-1} \nonumber \\ &=&\Id-(\widecheck{G}_Q+\widecheck{G}_Q^{\tr})\widecheck{G}_Q^{-1} \nonumber \\
& = & \Id-\Id-\widecheck{G}_Q^{\tr}\widecheck{G}_Q^{-1} \nonumber \\
&=& -\widecheck{G}_{q_Q}^{\tr}\widecheck{G}_{q_Q}^{-1} = \Cox_{q_Q}, \nonumber \vspace*{-3mm}
\end{eqnarray}
since $\widecheck{G}_Q=\widecheck{G}_{q_Q}$.
\end{proof}

We define the \textbf{Coxeter matrix} $\Cox_Q$ of a loop-less quiver $Q$ as
\[
\Cox_Q=\Id-I(Q)^{\tr}I(Q^{-1}).
\]
The Coxeter polynomial of $q_Q$, denoted by $\cox_Q(\va)$, is also referred to as \textbf{Coxeter polynomial} of $Q$.

\begin{corollary}\label{C:28}
Let $q:\Z^n \to \Z$ be a connected non-negative unit form of Dynkin type $\A_{n-c}$ (with $c$ the corank of $q$). Then the entries $c_{ij}$ of the Coxeter matrix $\Cox_q=(c_{ij})_{i,j=1}^n$ of $q$ are bounded as follows.
\[
|c_{ij}-\delta_{ij}|\leq 2, \quad \text{for $i,j=1,\ldots,n$},
\]
where $\delta_{ij}=1$ if $i=j$ and $\delta_{ij}=0$ otherwise. Moreover,
\begin{itemize}
 \item[a)] If $q$ is principal (that is, of corank one) and $|q_{ij}|\leq 1$ for all $i,j=1,\ldots,n$, then $|c_{ij}|\leq 2$ for $i,j=1,\ldots,n$.
 \item[b)] If $q$ is positive, then $|c_{ij}|\leq 1$ for $i,j=1,\ldots,n$.
\end{itemize}
\end{corollary}
\begin{proof}
By Proposition~\ref{P:19}, there is a connected loop-less quiver $Q$ such that $q=q_Q$. If $I_i$ and $I_i^{-1}$ denote respectively the columns of the (vertex-arrow) incidence matrices $I(Q)$ and $I(Q^{-1})$, then the $(i,j)$-th entry $d_{ij}$ of the matrix $\Cox_Q-\Id$ is $d_{ij}=-I_i^{\tr}I_j^{-1}$. In particular,
\[
|d_{ij}|\leq 2 \quad \text{and}\quad -2\leq d_{ii} \leq 0, \quad \text{for $i,j=1,\ldots,n$}.
\]
Hence the general claim follows since $d_{ij}+\delta_{ij}=c_{ij}$, by Theorem~\ref{T:27}.

\medskip
Let $q$ be a principal unit form. By Corollary~\ref{C:10}, we see that the connected quiver $Q$ satisfies $|Q_0|=|Q_1|$ (that is, $Q$ is a $1$-tree quiver). Assume, to the contrary, that there is an arrow $i$ such that $I_i^{\tr}I^{-1}_i=-2$, that is, $i$ and its corresponding arrow $i^*$ in the inverse quiver $Q^{-1}$ are parallel arrows with opposite directions, say
\[
\sou(i)=v=\tar^*(i^*) \quad \text{and} \quad \tar(i)=w=\sou^*(i^*),
\]
where $Q=(Q_0,Q_1,\sou,\tar)$ and $Q^{-1}=(Q_0,Q^*_1,\sou^*,\tar^*)$. Let $\alpha=\alpha_Q^-(i^{-1})$ and $\beta=\alpha_Q^-(i^{+1})$ be the minimally descending walks starting with arrow $i$, as given after Definition~\ref{D:22}. Then
\[
\sou(\alpha)=\tar(i)=\sou^*(i^*)=\tar(\alpha) \quad \text{and}\quad \sou(\beta)=\sou(i)=\tar^*(i^*)=\tar(\beta),
\]
that is, both $\alpha$ and $\beta$ are closed walks. Since $Q$ is a $1$-tree and $\alpha \neq \beta$ (for $\sou(\alpha) \neq \sou(\beta)$), and both $\alpha$ and $\beta$ are minimally descending walks, we conclude that $\alpha=i_0^{\epsilon_0}i_1^{\epsilon_1}$ and $\beta=i_0^{-\epsilon_0}i_1^{-\epsilon_1}$. In particular $|q_{i_0i_1}|=2$, which proves the claim~$(a)$.

\medskip
Assume now that $q$ is a positive unit form. Again by Corollary~\ref{C:10}, we see that the connected quiver $Q$ is a tree. Assume that $i$ and $j$ are arrows in $Q$ such that $\MiNu^*(i^*)=\MiNu(j)$. Let $\alpha=i_0^{\epsilon_0}i_1^{\epsilon_1}\cdots i_r^{\epsilon_r}$ and $\beta=j_0^{\eta_0}j_1^{\eta_1}\cdots j_s^{\eta_s}$ be as before, where $i=i_0=j_0$. In this situation, observe that there are signs $\epsilon, \eta \in \{ \pm 1\}$ such that the following is a non-trivial closed walk in $Q$,
\[
j_1^{\eta_1}\cdots j_s^{\eta_s}j^{\eta}i_r^{-\epsilon_r}\cdots i_1^{-\epsilon_1}i_0^{\epsilon},
\]
which is impossible since $Q$ is a tree. This shows that $|c_{ij}| \leq 1$ for $i \neq j$.

\medskip
Finally, for any arrow $i$ in $Q$ with corresponding arrow $i^*$ in $Q^{-1}$, we may argue as above to show that if $\MiNu(i)\cap \MiNu^*(i^*) \neq 0$, then either $\sou(i)=\sou^*(i^*)$ or $\tar(i)=\tar^*(i^*)$. This shows that $d_{ii}=-I_i^{\tr}I_{i}^{-1} \in \{-2,-1,0\}$, and in particular $c_{ii}=d_{ii}+1\in \{-1,0,1\}$, which completes the proof of $(b)$.
\end{proof}

\subsection{Extended maximal stars} \label{S(A):Star}

Recall that a \textbf{$1$-tree quiver} is a connected quiver $Q$ with $|Q_0|=|Q_1|$. By \textbf{maximal $1$-star} we mean a $1$-tree quiver $\widetilde{\Star}$ such there is a vertex $v \in \widetilde{\Star}_0$ (called \textbf{center of the star}) incident to all arrows of $\widetilde{\Star}$. Notice that there is exactly one pair of parallel arrows in $\widetilde{\Star}$, say $\ell < m$. In that case, if the maximal $1$-star quiver $\widetilde{\Star}$ has $n+1$ arrows and $1 \leq \ell<m \leq n+1$, we use the notation
\[
\widetilde{\Star}=\widetilde{\Star}^{\ell,m}_n.
\]
For instance, the cases $\widetilde{\Star}_4^{\ell,5}$ for $1 \leq \ell <5$, with corresponding inverses shown underneath, have the following shapes,
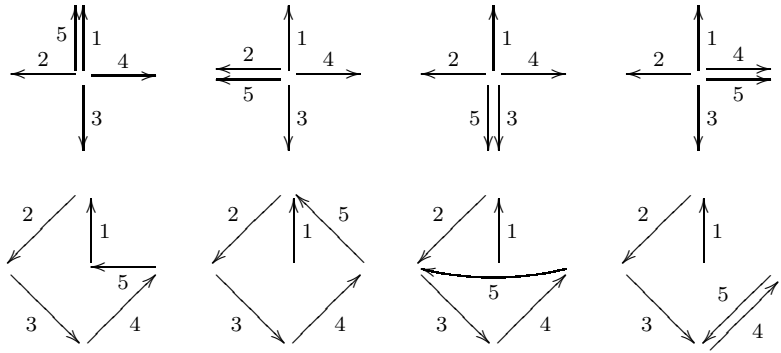
\vspace*{-1mm}
\[
\xymatrix{{} & {} \ar@{<-}@<-.6ex>[d]_5 \ar@{<-}[d]^-1 & {} \\ {}  \ar@{<-}[r]^-2 & {}  & {} \ar@{<-}@<.1ex>[l]_-4 \\ {} & {} \ar@{<-}[u]_-3  & {} } \quad \xymatrix{{} & {} \ar@{<-}[d]^-1  & {} \\ {}  \ar@{<-}@<-.4ex>[r]_-5 \ar@{<-}@<.4ex>[r]^-2 & {} & {} \ar@{<-}[l]_-4  \\ {} & {} \ar@{<-}[u]_-3  & {} } \quad  \xymatrix{{} & {}\ar@{<-}[d]^-1 & {} \\ {} \ar@{<-}[r]^-2 & {} & {} \ar@{<-}[l]_4  \\ {} & {}  \ar@{<-}@<-.4ex>[u]_-3 \ar@{<-}@<.4ex>[u]^-5 & {} } \quad \xymatrix{{} & {} \ar@{<-}[d]^-1  & {} \\ {}  \ar@{<-}[r]^-2 & {} & {}  \ar@{<-}@<.4ex>[l]^-5 \ar@{<-}@<-.4ex>[l]_-4 \\ {} & {} \ar@{<-}[u]_-3 & {} }
\]
\[
\xymatrix{{} & {}  \ar@{<-}@<.6ex>[d]^-1 & {} \\ {} \ar@{<-}[ru]^-2 & {} \ar@{<-}[r]_-5 & {}  \ar@{<-}[ld]^-4 \\ {} & {} \ar@{<-}[lu]^-3 & {} } \quad \xymatrix{{} & {} \ar@{<-}@<.4ex>[d]^-1 \ar@{<-}[rd]^-5 & {} \\ {} \ar@{<-}[ru]^-2  & {} & {} \ar@{<-}[ld]^-4 \\ {} & {}  \ar@{<-}[lu]^-3 & {} } \quad  \xymatrix{{} & {}\ar@{<-}@<.4ex>[d]^-1 & {} \\ {} \ar@{<-}[ru]^-2 \ar@{<-}@/_4pt/[rr]_-5  & {} & {}  \ar@{<-}[ld]^-4 \\ {} & {} \ar@{<-}[lu]^-3  & {} } \quad \xymatrix{{} & {}  \ar@{<-}@<.4ex>[d]^-1 & {} \\ {} \ar@{<-}[ru]^-2  & {} & {} \ar[ld]_-5   \\ {} & {} \ar@{<-}[lu]^-3 \ar@<-.8ex>[ru]_-4 & {} }
\]

\eject
We now generalize Lemma~\ref{L:16} and Proposition~\ref{P:17} to maximal $1$-stars.

\begin{lemma}\label{L:29}
Let $\widetilde{\Star}$ be a maximal $1$-star quiver.
\begin{itemize}
 \item[a)] For an arbitrary vertex $v$ of $\widetilde{\Star}$, there is a $\widetilde{\Star}$-admissible iterated $FS$-transformation $\FS$ such that $\widetilde{\Star}\FS$ is a maximal $1$-star with center $v$.
 \item[b)] If $\widetilde{\Star}=\widetilde{\Star}^{\ell,m}_n$ and $\widetilde{\Star}'=\widetilde{\Star}_n^{\ell',m'}$, then $q_{\widetilde{\Star}'} \approx q_{\widetilde{\Star}}$ if and only if
 \[
 (m'-\ell')=(m-\ell) \qquad \text{or} \qquad (m'-\ell')+(m-\ell)=n+1.
 \]
\end{itemize}
\end{lemma}
\begin{proof}
Let $\widetilde{\Star}_1=\{1,\ldots,{n},{n+1}\}$ be the arrows of the maximal $1$-star $\widetilde{\Star}$, all of them having in common the center of the star $v_0$, and assume $n > 3$. Take $\widetilde{\Star}=\widetilde{\Star}^{\ell,m}_n$ for arrows $1 \leq \ell < m \leq n+1$, and enumerate the non-central vertices of $\widetilde{\Star}$ so that $v_t$ is incident to arrow $t$ for $t=1,\ldots,m-1$, and to arrow ${t+1}$ for $t=m,\ldots,n$.

To prove $(a)$ it is enough to show, as in Lemma~\ref{L:16}, that there is a $\widetilde{\Star}$-admissible iterated $FS$ transformation $\FS$ such that $\widetilde{\Star}\FS$ is a maximal $1$-star with center $v_1$, and such that the vertex $v_2$ is incident to the minimal arrow of $\widetilde{\Star}\FS$. We distinguish three cases:

\medskip
\noindent \textbf{Case 1.} Assume first that $\ell>1$. We use the following iterated $FS$-transformation,
\[
\mathcal{W}=\FS_{2,1}\FS_{3,2}\cdots \FS_{n+1,n}.
\]
As in Lemma~\ref{L:16}, a direct computation shows
\[
\widetilde{\Star}^{\ell,m}_n\mathcal{W}=\widetilde{\Star}^{\ell-1,m-1}_n,
\]
where now $v_1$ is the star center of $\widetilde{\Star}\mathcal{W}$, and the first arrow of $\widetilde{\Star}\mathcal{W}$ is joining vertices $v_1$ and $v_2$.

\medskip
\noindent \textbf{Case 2.} Assume now that $\ell=1$ and $m>2$, and consider the following iterated $FS$ transformation,
\[
\mathcal{W}_m=\FS_{2,1}\FS_{3,2}\cdots\FS_{m-1,m-2}\FS_{m+1,m}\cdots \FS_{n+1,n},
\]
obtained from $\mathcal{W}$ by omitting the transformation $\FS_{m,m-1}$.
Notice similarly that $\widetilde{\Star}^{1,m}_n\mathcal{W}=\widetilde{\Star}^{m-1,n+1}_n$. Indeed, the quiver $Q:=\widetilde{\Star}\FS_{2,1}\FS_{3,2}\cdots\FS_{m-1,m-2}$ has the following shape,
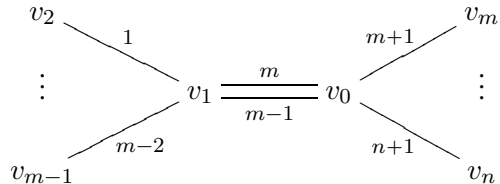
\[
\xymatrix@C=3pc@R=1pc{
{v_2} \ar@{-}[rd]^-{1} & & & {v_m} \ar@{-}[ld]_-{m+1} \\
\vdots & {v_1} \ar@{-}@<-.5ex>[r]_-{m-1} \ar@{-}@<.5ex>[r]^-{m} & {v_0} & \vdots \\
{v_{m-1}} \ar@{-}[ru]_-{m-2} & & & {v_n} \ar@{-}[lu]^-{n+1} }
\]
The arrows $m$ and $m-1$ are parallel in $Q$, therefore we omit $\FS_{m,m-1}$ to avoid loops (see Remark~\ref{R:06}). However, $\FS_{m+1,m}\cdots \FS_{n+1,n}$ is a $Q$-admissible iterated $FS$-transformation, and one can directly compute
\[
\widetilde{\Star}^{1,m}_n\mathcal{W}_m=Q\FS_{m+1,m}\cdots \FS_{n+1,n}=\widetilde{\Star}^{m-1,n+1}_n.
\]
Moreover, since $m>2$, both vertices $v_1$ and $v_2$ are incident to the minimal arrow of $\widetilde{\Star}\mathcal{W}_m$.

\medskip
\noindent \textbf{Case 3.} Assume now that $\ell=1$ and $m=2$. Take
\[
\mathcal{M}=\FS_{n,n+1}\FS_{n-1,n}\cdots \FS_{1,2} \quad \text{and} \quad\mathcal{M}_{n}=\FS_{n-1,n}\cdots\FS_{1,2},
\]
and observe that $Q'=\widetilde{\Star}\mathcal{M}^{n-1}$ has the following shape,
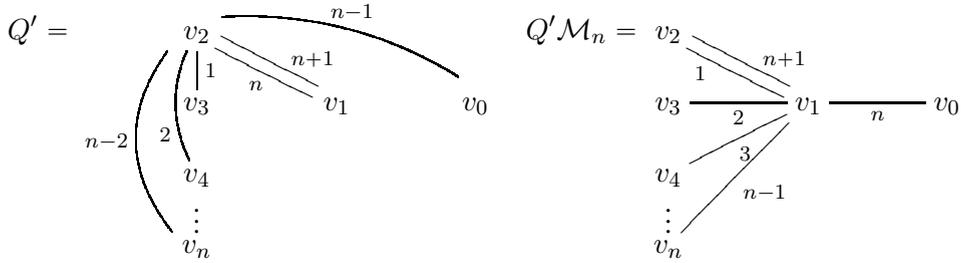
\[
Q'=\xymatrix@C=3pc@R=1.2pc{
{v_2} \ar@{-}[d]^-1 \ar@{-}@<-.5ex>[rd]_-n \ar@{-}@<.5ex>[rd]^(.7){n+1} \ar@{-}@/^10pt/@<1ex>[rrd]^-{n-1} \ar@{-}@/_8pt/[dd]_(.7){2} \ar@{-}@<-1ex>@/_18pt/[ddd]_-{n-2} \\ {v_3} & {v_1}  & {v_0} \\
{v_4} \ar@{}[d]|-{\vdots} \\ v_{n} } \quad  Q'\mathcal{M}_n= \xymatrix@C=3pc@R=1.2pc{
{v_2} \ar@{-}@<-.5ex>[rd]_(.3){1} \ar@{-}@<.5ex>[rd]^(.7){n+1}  \\ {v_3} \ar@{-}[r]_-2 & {v_1} \ar@{-}[r]_-{n} & {v_0} \\
{v_4} \ar@{-}[ru]_-3 \ar@{}[d]|-{\vdots} \\ v_{n} \ar@{-}[ruu]_-{n-1} }
\]
Notice also that $Q'\mathcal{M}_n=\widetilde{\Star}^{1,n+1}_n$ is a maximal $1$-star with center $v_1$, and such that $v_2$ is incident to the minimal arrow of $Q'\mathcal{M}_n$.

\medskip
Take now $\FS$ to be the transformation $\mathcal{W}$, $\mathcal{W}_m$ or $\mathcal{M}^{n-1}\mathcal{M}_n$ in cases~1, 2 and~3 respectively, and observe that we have
\begin{equation} \label{EqThree}
\widetilde{\Star}^{\ell,m}_n\FS= \left\{ \begin{array}{l l} \widetilde{\Star}^{\ell-1,m-1}_n, & \text{if $\ell>1$},\\ \widetilde{\Star}^{m-1,n+1}_n, & \text{if $\ell=1$}.\end{array} \right.
\end{equation}
By the above, $\widetilde{\Star}\FS$ is a maximal $1$-star with center $v_1$ and such that $v_2$ is incident to the minimal arrow of $\widetilde{\Star}\FS$, which shows the inductive step to complete the proof of $(a)$.

\medskip
To show $(b)$, if
 \[
 (m'-\ell')=(m-\ell) \qquad \text{or} \qquad (m'-\ell')+(m-\ell)=n+1,
 \]
 then by Corollary~\ref{C:15} and equation~(\ref{EqThree}) we have $q_{\widetilde{\Star}'} \approx q_{\widetilde{\Star}}$. For the converse, assume that $q_{\widetilde{\Star}'} \approx q_{\widetilde{\Star}}$. Using equation~(\ref{EqThree}), we may also assume that $m=n+1=m'$, and that $1 \leq \ell,\ell' \leq \frac{n+1}{2}$. In this case, consider the shape of the Coxeter polynomial of $\widetilde{\Star}$ (see remark below), must have $\ell=\ell'$. Hence the result.
\end{proof}

\begin{remark}\label{R:29medio}
A direct calculation using the description of Coxeter matrices in Theorem~\ref{T:27} yields the following Coxeter polynomial for a maximal $1$-star $\widetilde{\Star}=\widetilde{\Star}^{\ell,m}_n$,
\[
\varphi_{\widetilde{\Star}}(\va)=(\va^{m-\ell}-1)(\va^{(n+1)-(m-\ell)}-1).
\]
Therefore, Lemma~\ref{L:29}$(b)$ may be reinterpreted as follows:
\begin{itemize}
\item[b')] Two maximal $1$-star quivers are strongly Gram congruent if and only if they have the same Coxeter polynomial.
\end{itemize}
Results of this kind for posets may be found in~\cite{GSZ14}.
\end{remark}

\begin{proposition}\label{P:30}
Let $Q$ be a $1$-tree quiver. Then for any vertex $v$ in $Q$ there is a $Q$-admissible iterated $FS$-transformation $\FS$ such that $Q\FS$ is a maximal $1$-star with center the vertex $v$.
\end{proposition}
\begin{proof}
We proceed as in Proposition~\ref{P:17}, that is, by induction on the number $n=|Q_1|$ of arrows in $Q$. For $n=1,2$ the tree $Q$ is a $1$-star, and we may change the position of its center as in the Lemma above. Hence, we may assume that $n \geq 3$ and that the claim holds for all $1$-trees with less than $n$ arrows.

Let $n$ be the maximal arrow in $Q$ (relative to the total order $\leq$ in $Q_1$) and take $\MiNu(n)=\{v,w\}$. Let $Q'$ be the quiver obtained from $Q$ by deleting the arrow $n$. The set $Q'_1$ inherits a total order from $Q_1$.  Observe that, by the maximality of $n$, any $Q'$-admissible iterated $FS$-transformation is also $Q$-admissible. We distinguish two cases:

\medskip
\noindent \textbf{Case 1.} Assume first that $Q'$ is a connected quiver. Then $Q'$ is a ($0-$)tree, and we may use Proposition~\ref{P:17} to assume that $Q'$ is a maximal star with center $v$. Since $\MiNu(n)=\{v,w\}$, then $Q$ is a maximal $1$-star.

\medskip
\noindent \textbf{Case 2.} Assume now that $Q'$ is not connected. Then $Q'$ is the disjoint union of exactly two connected quivers, one containing vertex $v$ and denoted by $Q^v$, and one containing vertex $w$ and denoted by $Q^w$.

\medskip
\textbf{Subcase 2.1.} Assume first that $|Q^w_0|=1$. Then $Q^v$ is a $1$-tree, and by induction hypothesis we may assume that $Q^v$ is a maximal $1$-star with center $v$. Hence $Q$ is a maximal star, and we proceed analogously if $|Q^v_0|=1$.

\medskip
\textbf{Subcase 2.2.} Assume now that $|Q^w_0|>1$ and $|Q^v_0|>1$, and that the second largest arrow $n-1$ in $Q$ belongs to $Q^v$. Since $Q$ is a $1$-tree, either $Q^v$ is a $0$-tree or a $1$-tree with less than $n$ arrows. Thus, by Proposition~\ref{P:17} or induction hypothesis respectively, we may assume that $Q^v$ is a maximal $c$-star with center $v$ for $c=0$ or $c=1$. First, if $n-1$ is a pendant arrow in $Q^v$, then $n$ is a pendant arrow in $Q'=Q\FS^{\epsilon}_{n-1,n}$, and we may apply Case~1 above to the $1$-tree quiver $Q'$. Second, if $n-1$ has a parallel arrow $j$ in $Q^{v}$, then the arrows $j$, $n-1$ and $n$ form a cycle in $Q'$, and we may apply Case~1 above to the $1$-tree quiver $Q'$.

\medskip
\noindent ~~~~To complete the proof, use Lemma~\ref{L:29}$(a)\,$to change the center of the resulting$\,1$-star as desired.
\end{proof}

We end this section with the second main classification result of the paper.

\begin{theorem}\label{T:34}
Two (connected) principal unit forms of Dynkin type $\A_n$ are strongly Gram congruent if and only if they have the same Coxeter polynomial.
\end{theorem}

\begin{proof}
Since Coxeter polynomials are strong Gram invariants (Lemma~\ref{L:26}), we only need to show sufficiency: assume $q$ and $q'$ are principal unit forms in $n+1$ variables, both of Dynkin type $\A_n$ and same Coxeter polynomial.  By Proposition~\ref{P:19}, there are connected $1$-tree quivers $Q$ and $Q'$ such that $q=q_Q$ and $q'=q_{Q'}$. By Proposition~\ref{P:30}, there is a $Q$-admissible iterated $FS$-transformation $\FS$ and a maximal $1$-star $\widetilde{\Star}_n^{\ell,m}$ (for some $n \geq 1$ and $1 \leq \ell < m \leq n+1$), such that $Q\FS=\widetilde{\Star}_n^{\ell,m}$. By Lemma~\ref{L:29}$(b)$ and Remark~\ref{R:29medio}, we may assume that $m=n+1$ and $1 \leq \ell \leq \frac{n+1}{2}$ is such that
\[
\varphi_{Q}(\va)=\varphi_{\widetilde{\Star}_n^{\ell,n+1}}(\va)=(\va^{\ell}-1)(\va^{n+1-\ell}-1).
\]

Proceeding similarly for $q_{Q'}$, since $\varphi_{Q'}=\varphi_{Q}$, we find a $Q'$-admissible iterated $FS$-transformation $\FS'$ with $Q'\FS'=\widetilde{\Star}_n^{\ell,n+1}=Q\FS$, hence $q \approx q'$ by Corollary~\ref{C:15}.
\end{proof}

\section*{Concluding remarks and future work}

We stress that the proof of all preparatory results towards the main theorems (the elementary quiver transformations, Lemmas~\ref{L:16},~\ref{L:29}, Propositions~\ref{P:17},~\ref{P:30}, and most importantly, Proposition~\ref{P:19}) are completely constructive, and can be easily implemented in any programming language of general use. In particular, one may follow the proofs of Theorems~\ref{T:20} and~\ref{T:34} to find algorithmic solutions to Problem~B in terms of iterated $FS$-transformations, for the case of non-negative unit forms of Dynkin type $\A_n$ of corank zero or one.

\medskip
It seems to be a good idea to consider quadratic forms $q:\Z^n \to \Z$ having symmetric Gram matrix $G_q$ factorized as
\[
G_q=I^{\tr}I,
\]
for a $n \times m$ matrix $I$ with ``special'' properties, for instance, one having columns in the root systems as given in~\cite[Definitions~3.1 and~3.2]{CGSS}. Any such quadratic form is clearly non-negative. Here we consider the root system $A_n$ given in~\cite[Definition~3.1]{CGSS}, that is, matrices $I$ such that for each column $I\bas_i$ (for $i=1,\ldots,n$) there are signs $S,T \in \{\pm 1\}$ and indices $s,t \in \{1,\ldots,n\}$ satisfying
\begin{itemize}
\itemsep=0.9pt
 \item[i)] $I\bas_i=S\bas_s+T\bas_t$.
 \item[ii)] $S \neq T$.
 \item[iii)] $s \neq t$.
\end{itemize}
Indeed, matrices with these three conditions are precisely the (vertex-arrow) incidence matrices of loop-less quivers, and Proposition~\ref{P:19} asserts that the corresponding quadratic forms are precisely the non-negative unit forms with all components of Dynkin type $\A_r$.

\medskip
Assume, additionally, that $A$ is a $\Z$-invertible matrix morsification of $q$ with integer coefficients (in the sense of Simson~\cite{dS13a}). As in the proof of Theorem~\ref{T:27}, the Coxeter matrix $\Cox_A=-A^{\tr}A^{-1}$ of $A$ admits the following expression,
\[
\Cox_A=\Id_n-I^{\tr}IA^{-1}.
\]
In an upcoming paper~\cite{jaJ2020b}, we show that the similarity invariants of $\Cox_A$ correspond to the orthogonal invariants of the matrix
\[
\Lambda_A=\Id_m-IA^{-1}I^{\tr},
\]
which turns out to be an orthogonal matrix, producing in this way many important strong Gram invariants of $A$. The particular case when $IA^{-1}$ satisfies again conditions $(i-iii)$ above has special combinatorial features, as illustrated in Section~\ref{S(A):InvDos} when $A$ is the standard morsification of $q$. In this case, $IA^{-1}$ is precisely the incidence matrix of the inverse of the quiver with incidence matrix $I$. Although successful for coranks zero and one, we do not know whether the technique of $FS$-transformations used in Lemmas~\ref{L:16} and~\ref{L:29} can be generalized to higher coranks (even corank two). As in the proofs of Propositions~\ref{P:17} and~\ref{P:30}, and Theorems~\ref{T:20} and~\ref{T:34}, such a generalization would imply the Coxeter spectral determination of strong Gram classes of non-negative unit forms of Dynkin type $\A_n$, as in the positive and principal case. In an upcoming work, we approach such strong classification with a matricial method.

\medskip
A different direction is to omit conditions $(ii)$ and $(iii)$ on $I$, that is, to consider simply ``incidence matrices'' as in~\cite{tZ08}. In a future work we show that the corresponding quadratic forms include not only non-negative semi-unit forms of Dynkin type $\A_n$, but also those of Dynkin type $\D_n$, as well as the Euler form of important classes of algebras (for instance, gentle algebras of finite global dimension). Moreover, the results of~\ref{S(A):Cox} and~\cite{jaJ2020b} can be extended to unimodular morsifications of quadratic forms with Gram matrix factorized by ``incidence matrices'', potentially facilitating their Coxeter spectral analysis.

In some sense, the matrix $I$ \emph{exposes} internal properties of $I^{\tr}I$. The straightforward constructions and considerations of the paper work in support of this claim.

\bibliographystyle{fundam}

\end{document}